\documentclass{amsart}
\usepackage{amssymb}

\input xy
\xyoption{all}

\newcommand{\C}{\mathcal C}
\newcommand{\DD}{\mathcal D}
\newcommand{\HH}{\mathcal H}
\newcommand{\F}{\mathcal F}
\newcommand{\U}{\mathcal U}
\newcommand{\V}{\mathcal V}
\newcommand{\W}{\mathcal W}
\newcommand{\w}{\omega}
\newcommand{\e}{\varepsilon}
\newcommand{\II}{\mathbb I}
\newcommand{\IN}{\mathbb N}
\newcommand{\IR}{\mathbb R}
\newcommand{\A}{\mathcal A}
\newcommand{\K}{\mathcal K}
\newcommand{\BB}{\mathcal B}
\newcommand{\id}{\mathrm{id}}
\newcommand{\St}{\mathcal{S}t}
\newcommand{\Int}{\mathrm{Int}}
\newcommand{\setmap}{\multimap}
\newcommand{\diam}{\mathrm{diam}}
\newcommand{\Ra}{\Rightarrow}
\newcommand{\dist}{\mathrm{dist}}

\newtheorem{theorem}{Theorem}[section]
\newtheorem{corollary}[theorem]{Corollary}
\newtheorem{proposition}[theorem]{Proposition}
\newtheorem{lemma}[theorem]{Lemma}
\newtheorem{claim}[theorem]{Claim}
\theoremstyle{definition}
\newtheorem{definition}[theorem]{Definition}
\newtheorem{remark}[theorem]{Remark}

\title{Universal nowhere dense subsets of locally compact manifolds}
\author{Taras Banakh and Du\v san Repov\v s}
\address{T.Banakh: Jan Kochanowski University in Kielce (Poland) and Ivan Franko National University of Lviv (Ukraine)}
\email{t.o.banakh@gmail.com}
\address{D.Repovs: Faculty of Mathematics and Physics, and
Faculty of Education, University of Ljubljana (Slovenia)}
\email{dusan.repovs@guest.arnes.si}
\keywords{Universal nowhere dense subset, Sierpi\'nski carpet, Menger cube, Hilbert cube manifold,  $n$-manifold, tame ball, tame decomposition}
\subjclass[2010]{57N20, 57N45, 57N60}

\begin{document}
\begin{abstract} In each manifold $M$ modeled on a finite or infinite dimensional cube $[0,1]^n$, $n\le \w$, we construct a closed nowhere dense subset $S\subset M$ (called a spongy set) which is a universal nowhere dense set in $M$ in the sense that for each nowhere dense subset $A\subset M$ there is a homeomorphism $h:M\to M$ such that $h(A)\subset S$. The key tool in the construction of  spongy sets is a theorem on topological equivalence of certain decompositions of manifolds. A special
case of this theorem says that two vanishing cellular strongly shrinkable decompositions $\A,\BB$ of a Hilbert cube manifold $M$ are topologically equivalent if any two non-singleton elements $A\in\A$ and $B\in\BB$ of these decompositions are ambiently homeomorphic.
\end{abstract}
\maketitle

\section{Introduction} In this paper we shall construct and characterize universal nowhere dense subsets of manifolds modeled on finite or infinite dimensional cubes $\II^n$, $n\le \w$. A paracompact space $M$ is called a {\em manifold modeled on a model space $E$} (briefly, an {\em $E$-manifold\/}) if  each point $x\in X$ has an open neighborhood $O_x\subset M$ homeomorphic to an open subset of the model space $E$.

A nowhere dense subset $N$ of a topological space $M$ is called {\em a universal nowhere dense set} in $M$ if for each nowhere dense subset $A\subset M$ there is a homeomorphism $h:M\to M$ such that $h(A)\subset N$.

It is well-known that the standard Cantor set $M^1_0$ is a universal nowhere dense subset of the unit interval $\II=[0,1]$ and the Sierpi\'nski carpet $M^2_1$ is a universal nowhere dense subset of the square $\II^2$. The Cantor set and the Sierpi\'nski carpet are first representatives in the hierarchy of the Menger cubes $M^n_{n-1}$, which are universal nowhere dense subsets of the $n$-dimensional cubes $\II^n$, see  \cite{Menger}.

The topology of the pair $(\II^2,M^2_1)$ was characterized by Whyburn \cite{Whyburn}.  His result was generalized by Cannon \cite{Cannon73} who gave a topological characterization of the pair $(\II^n,M^n_{n-1})$ for all positive integers
$n\ne 4$.
In this paper we shall generalize these results of Whyburn and Cannon by constructing a specific universal nowhere dense subset $S$ (called a {\em spongy set}) in each $\II^n$-manifold  $M$ and giving a topological characterization of the resulting pair $(M,S)$. The definition of a spongy set is based on the notion of a tame ball.

\begin{definition} A subset $B$ of an $\II^n$-manifold $M$, $n\le \w$, is called a {\em tame ball\/}  in $M$ if $B$ has an open neighborhood $O(B)\subset M$ such that the pair $(O(B),B)$ is homeomorphic to the pair
$$\begin{cases}
(\IR^n,\II^n)&\mbox{if $n<\w$,}\\
\big(\II^\w\times[0,2),\II^\w\times[0,1]\big)&\mbox{if $n=\w$}.
\end{cases}
$$
\end{definition}

A family $\F$ of subsets of a topological space $X$ is called {\em vanishing} if for any open cover $\U$ of $X$ the family $\F'=\{F\in\F:\forall U\in\U\;\;F\not\subset U\}$ is {\em locally finite} in $X$.

\begin{definition} A subset $S$ of an $\II^n$-manifold $M$, $n\le \w$, is called a {\em spongy set} in $M$ if
\begin{enumerate}
\item $S$ is closed and nowhere dense in $M$,
\item the family $\C$ of connected components of the complement $M\setminus S$ is vanishing in $M$,
\item any two connected components $C,C'\in\C$ have disjoint closures in $M$, and
\item the closure $\bar C$ of each connected component $C\in\C$ is a tame ball in $M$.
\end{enumerate}
\end{definition}

A typical example of a spongy set in a finite dimensional cube $\II^n$ is the Menger cube $M^n_{n-1}$. The following theorem generalizes the results of Whyburn \cite{Whyburn} (for $n=2$)
and Cannon \cite{Cannon73} (for $n\in\IN\setminus\{4\}$) and gives many examples of universal nowhere dense subsets in finite and infinite dimensional manifolds. This theorem is essentially used in the paper \cite{BR} devoted to constructing universal meager subsets in locally compact manifolds.

\begin{theorem}\label{main1} Let $M$ be a manifold modeled on a cube $\II^n$, $n\le\w$.
\begin{enumerate}
\item Each nowhere dense subset of $M$ lies in a spongy subset of $M$.
\item Any two spongy subsets of $M$ are ambiently homeomorphic.
\item Any spongy subset of $M$ is a universal nowhere dense subset in $M$.
\end{enumerate}
\end{theorem}

Two subsets $A,B$ of a topological space $X$ are called {\em ambiently homeomorphic} if the pairs $(X,A)$ and $(X,B)$ are homeomorphic. The latter means that $h(A)=B$ for some homeomorphism $h:X\to X$.

The spongy subsets $M^n_{n-1}$ of finite dimensional cubes $\II^n$ are typical examples of deterministic fractals (see \cite{Barnsley} and \cite{Falconer} for the theory of fractals). In contrast, spongy sets in Hilbert cube manifolds do not have such a fractal structure and they are Hilbert cube manifolds as well.

\begin{theorem}\label{t:spongeQ} Any spongy subset $S$ of a Hilbert cube manifold $M$ is a retract of $M$ and is homeomorphic to $M$.
\end{theorem}

This theorem will be proved in Section~\ref{s:spongeQ}.
Theorem~\ref{main1} will be proved in Section~\ref{s:main1} after long preparatory work in Sections~\ref{s:dec}--\ref{s:main2}. The principal tool in the proof of Theorem~\ref{main1} is Theorem~\ref{main2} on the topological equivalence of $\K$-tame decompositions of strongly locally homogeneous completely metrizable spaces, discussed in Section~\ref{s:dec} and proved in Section~\ref{s:main2}. In Section~\ref{s:cellular} we shall apply Theorem~\ref{main2} to prove Corollaries~\ref{c4.8} and \ref{c4.9} establishing the topological equivalence of some vanishing cellular decompositions of Hilbert cube manifolds.

\section{Topological equivalence of certain decompositions of topological spaces}\label{s:dec}

In this section we discuss the problem of the topological equivalence of decompositions of completely metrizable spaces. For the theory of decompositions of finite dimensional manifolds we refer the reader to Daverman's monograph \cite{Dav}. Now let us fix some notation.

For a subset $A$ of a topological space $X$ we shall denote by $\bar A$, $\Int(A)$, and $\partial A=\bar A\setminus\Int(A)$  the closure, the  interior, and the boundary  of $A$ in $X$, respectively. For a metric space $(X,d)$, a point $x\in X$ and a subset $A\subset X$ we put $d(x,A)=\inf_{a\in A}d(x,a)$ and $\diam(A)=\sup\{d(a,b):a,b\in A\}$. For a real number $\e$ we shall denote by $O_d(x,\e)=\{y\in X:d(x,y)<\e\}$ and $O_d(A,\e)=\{x\in X:d(x,A)<\e\}=\bigcup_{a\in A}O_d(a,\e)$  the open $\e$-neighborhoods of $a$ and $A$ in the metric space $X$.

Let $\A,\BB$ be two families $\A,\BB$ of subsets of a space $X$. We shall write $\A\prec\BB$ and say that the family $\A$ refines the family $\BB$ if each set $A\in\A$ is contained in some set $B\in\BB$.

A subset $A\subset X$ is called {\em $\BB$-saturated} if $A$ coincides with its {\em $\BB$-star} $\St(A,\BB)=\bigcup\{B\in\BB:A\cap B\ne \emptyset\}$. The family $\A$ is called {\em $\BB$-saturated} if each set $A\in\A$ is $\BB$-saturated. The family $\St(\A,\BB)=\{\St(A,\BB):A\in\BB\}$ will be called the {\em $\BB$-star} of the family $\A$, and $\St(\A)=\St(\A,\A)$ is the {\em star} of $\A$.

Given functions $f,g:Z\to X$ we write $(f,g)\prec\A$ if for each point $z\in Z$ with $f(z)\ne g(z)$ the doubleton $\{g(z),f(z)\}$ lies in some set $A\in\A$. This definition implies that $f(z)=g(z)$ for each point $z\in Z\setminus \big(f^{-1}(\bigcup\A)\cap g^{-1}(\bigcup\A)\big)$. If $d$ is a metric on the space $X$, then we denote by $d(f,g)=\sup_{z\in Z}d(f(z),g(z))$  the $d$-distance between the functions $f,g$. Sometimes by $d(f,g)$ we shall also understand the function $d(f,g):X\to \IR$, $d(f,g):x\mapsto d(f(x),g(x))$.

A topological space $X$ is called {\em completely metrizable} if its topology is generated by a complete metric. By \cite[4.3.26]{En}, a topological space is completely metrizable if and only if it is metrizable and \v Cech complete. It is well-known \cite[5.1.8]{En} that each metrizable space $X$ is {\em collectionwise normal} in the sense that for each discrete family $\F$ of closed subsets of $X$ there is a discrete family $\{U_F\}_{F\in\F}$ of open subsets of $X$ such that $F\subset U_F$ for all $F\in\F$.

By a {\em decomposition} of a topological space $X$ we understand a cover $\DD$ of $X$ by pairwise disjoint non-empty compact subsets. For each decomposition $\DD$ we can consider the quotient map $q_\DD:X\to \DD$ assigning to each point $x\in X$ the unique compact set $q(x)\in\DD$ that contains $x$. The quotient map $q_\DD$ induces the quotient topology on $\DD$ turning $\DD$ into a topological space called the {\em decomposition space} of the decomposition $\DD$. Sometimes to distinguish a decomposition $\DD$ from its decomposition space we shall denote the latter space by $X/\DD$.

A decomposition $\DD$ of a topological space $X$ is {\em upper semicontinuous} if for each closed subset $F\subset X$ its {\em $\DD$-saturation} $\St(F,\DD)=\bigcup\{D\in\DD:D\cap F\ne\emptyset\}$ is closed in $X$. It is easy to see that a decomposition $\DD$ of $X$ is upper semicontinuous if and only if the quotient map $q_\DD:X\to X/\DD$ is closed if and only if the quotient map $q_\DD$ is perfect (the latter means that $q_\DD$ is closed and for each point $y\in X/\DD$ the preimage $q_\DD^{-1}(y)$ is compact). Since the (complete) metrizability is preserved by perfect maps (see \cite[3.9.10 and 4.4.15]{En}), we get the following lemma (cf. Proposition 2 \cite{Dav}).

\begin{lemma}\label{l1} For any upper semicontinuous decomposition $\DD$ of a (completely) metrizable space $X$ the decomposition space $X/\DD$ is (completely) metrizable.
\end{lemma}

Let us recall that a decomposition $\DD$ of a topological space $X$ is called  {\em vanishing} if for each open cover $\U$ of $X$ the subfamily $\DD'=\{D\in\DD:\forall U\in\U\;\;D\not\subset U\}$ is discrete in $X$ in the sense that each point $x\in X$ has a neighborhood $O_x\subset X$ that meets at most one set $D\in\DD'$.

Each  vanishing disjoint family $\C$ of non-empty compact subsets of a topological space $X$ generates the  vanishing decomposition
$$\dot\C=\C\cup\big\{\{x\}:x\in X\setminus\textstyle{\bigcup}\C\big\}$$of the space $X$.
In particular, each non-empty compact set $K\subset X$ induces the vanishing decomposition $\{K\}\cup\big\{\{x\}:x\in X\setminus K\}$ whose decomposition space will be denoted by $X/K$.
By $q_K:X\to X/K$ we shall denote the corresponding quotient map.

The following (probably known) lemma generalizes Proposition 3 of \cite{Dav}.

\begin{lemma}\label{l2} Each vanishing decomposition $\DD$ of a regular space $X$ is upper semicontinuous.
\end{lemma}

\begin{proof} Given a closed subset $F\subset X$ we need to check that its $\DD$-saturation $\St(F,\DD)=q_\DD^{-1}(q_\DD(F))$ is closed in $X$. Fix any point $x\in X\setminus \St(F,\DD)$ and let $D_x=q_\DD(x)$ be the unique element of the decomposition $\DD$, which contains the point $x$. By the regularity of the space $X$, the compact subset $D_x\subset X\setminus F$ has an open neighborhood $V\subset X$ such that $\overline{V}\cap F=\emptyset$. Since the decomposition $\DD$ is vanishing, for the open cover $\U=\{X\setminus F,X\setminus \overline{V}\}$ of $X$ the family $$\DD'=\{D\in\DD:D\not\subset X\setminus F,\;D\not\subset X\setminus \overline{V}\}=\{D\in\DD:D\cap F\ne\emptyset\ne D\cap\overline{V}\}$$ is discrete in $X$ and hence its union $D'=\bigcup\DD'$ is closed in $X$. Since $D_x\notin \DD'$, we conclude that $D_x\cap D'=\emptyset$ and hence $U_x=V\setminus D'$ is an open neighborhood of $x$ missing the set $\St(F,\DD)$ and therefore the latter set is closed in $X$.
\end{proof}

A decomposition $\DD$ of a space $X$ will be called {\em dense} (resp. {\em discrete}) if its non-degeneracy part $$\DD^\circ=\{D\in\DD:|D|>1\}$$ is dense (resp. closed and discrete) in the decomposition space $\DD=X/\DD$.

A decomposition $\DD$ of a topological space $X$ is called
\begin{itemize}
\item {\em shrinkable} if for each $\DD$-saturated open cover $\U$ of $X$ and each open cover $\V$ of $X$ there is a homeomorphism $h:X\to X$ such that $(h,\id_X)\prec\U$ and $\{h(D):D\in\DD\}\prec\V$;
\item {\em strongly shrinkable} if for each $\DD$-saturated open set $U\subset X$ the decomposition $\DD|U=\{D\in\DD:D\subset U\}$ of $U$ is shrinkable.
\end{itemize}
A compact subset $K$ of a topological space $X$ is called {\em locally shrinkable}  if for each neighborhood $O(K)\subset X$ and any open cover $\V$ of $O(K)$ there is a homeomorphism $h:X\to X$ such that $h|X\setminus O(K)=\id$ and $h(K)$ is contained in some set $V\in\V$. It is easy to see that a compact subset $K\subset X$ is locally shrinkable if and only if  the decomposition $\{K\}\cup\big\{\{x\}:x\in X\setminus K\big\}$ of $X$ is strongly shrinkable (cf. \cite[p.42]{Dav}).

(Strongly) shrinkable decompositions are tightly connected with (strong) near homeomorphisms.

A map $f:X\to Y$ between topological spaces will be called a
\begin{itemize}
\item a {\em near homeomorphism} if for each open cover $\U$ of $Y$ there is a homeomorphism $h:X\to Y$ such that $(h,f)\prec\U$;
\item a {\em strong near homeomorphism} if for each open set $U\subset Y$ the map $f|f^{-1}(U):f^{-1}(U)\to U$ is a near homeomorphism.
\end{itemize}

The following Shrinkability Criterion was proved in \cite[Theorem 2.6]{Dav}.

\begin{theorem}[Shrinkability Criterion]\label{t:shrink} An upper semicontinuous decomposition $\DD$ of a completely metrizable space $X$ is (strongly) shrinkable if and only if the quotient map $q_\DD:X\to X/\DD$ is a (strong) near homeomorphism.
\end{theorem}

For two decompositions $\A\prec\BB$ of a space $X$  we shall denote by $q^\A_\BB:X/\A\to X/\BB$ the unique map making the following diagram commutative:
$$\xymatrix{
&X\ar[rd]^{q_\BB}\ar[ld]_{q_\A}&\\
X/\A\ar[rr]_{q^\A_\BB}&&X/\BB
}$$

We shall say that a decomposition $\A$ of a topological space $X$ is {\em topologically equivalent} to a decomposition $\BB$ of a topological space $Y$ if there is a homeomorphism $\Phi:X\to Y$ such that the
decomposition $\Phi(\A)=\{\Phi(A):A\in\A\}$ of $Y$ is equal to the decomposition $\BB$. This happens if and only if there is a unique homeomorphism $\varphi:X/\A\to Y/\BB$ making the diagram
$$\xymatrix{
X\ar[r]^{\Phi}\ar[d]_{q_\A}&Y\ar[d]^{q_\BB}\\
X/\A\ar[r]_{\varphi}&Y/\BB
}$$commutative. In this case we shall say that the homeomorphism $\Phi$ is  $(\A,\BB)$-factorizable and the homeomorphism $\varphi:X/\A\to Y/\BB$ is $(\A,\BB)$-liftable.

More precisely, we define a homeomorphism
$\varphi:X/\A\to Y/\BB$ (resp. $\Phi:X\to Y$) to be {\em $(\A,\BB)$-liftable} (resp. $(\A,\BB)$-{\em factorizable}~) if there is a homeomorphism $\Phi:X\to Y$ (resp. $\varphi:X/\A\to Y/\BB$~) such that
$q_\BB\circ \Phi=\varphi\circ q_\A$. It is clear that each $(\A,\BB)$-liftable homeomorphism $\varphi:X/\A\to Y/\BB$ maps the non-degeneracy part $\A^\circ$ of the decomposition $\A$ onto the non-degeneracy part $\BB^\circ$ of the decomposition $\BB$. So, $\varphi:(\A,\A^\circ)\to (\BB,\BB^\circ)$ is a homeomorphism of pairs.

Observe that two decompositions $\A,\BB$ of a topological space $X$ are  topologically equivalent if and only if there is an $(\A,\BB)$-factorizable homeomorphism $\Phi:X\to X$ if and only if there exists an $(\A,\BB)$-liftable homeomorphism $\varphi:X/\A\to X/\BB$ between the decomposition spaces.

We shall be interested in finding conditions on vanishing decompositions $\A$, $\BB$ of a space $X$, which guarantee that the set of $(\A,\BB)$-liftable homeomorphisms is dense in the space $\HH(\A,\BB)$ of all homeomorphisms between the decomposition spaces $\A=X/\A$ and $\BB=X/\BB$.

The homeomorphism space $\HH(\A,\BB)$ will be endowed with the {\em limitation topology} \cite{Chig} whose neighborhood base at a homeomorphism $f:X\to Y$ consists of the sets
$$N(f,\U)=\{g\in \HH(X,Y):(f,g)\prec\U\}$$ where $\U$ runs over all open covers of $Y$.

The following definition of a tame family will be used in Definition~\ref{d:d-Ktame} of a $\K$-tame decomposition.

\begin{definition}\label{d:K-tame} Let $\K$ be a family of compact subsets of a topological space $X$.
We shall say that the family $\K$
\begin{itemize}
\item is {\em ambiently invariant} if for each homeomorphism $h:X\to X$ and each set $K\in\K$ we get $h(K)\in\K$;
\item has the {\em local shift property} if for any point $x\in X$ and a neighborhood $O_x\subset X$ there is a neighborhood $U_x\subset O_x$ of $x$ such that for any sets $A,B\in\K$ with $A,B\subset U_x$ there is a homeomorphism $h:X\to X$ such that $h(A)=B$ and $h|X\setminus O_x=\id|X\setminus O_x$;
\item {\em tame} if $\K$ is ambiently invariant, consists of locally shrinkable sets, has the local shift property, and each non-empty open subset $U\subset X$ contains a set $K\in\K$.
\end{itemize}
\end{definition}

Now we can define $\K$-tame decompositions.

\begin{definition}\label{d:d-Ktame} Let $\K$ be a tame family of compact subsets of a Polish space $X$. A decomposition $\DD$ of $X$ is called {\em $\K$-tame} if $\DD$ is vanishing, strongly shrinkable, and $\DD^\circ\subset\K$.
\end{definition}

The following theorem that will be proved in Section~\ref{s:exist} yields many examples of $\K$-tame decompositions.

\begin{theorem}\label{t:exist} Let $\K$ be a tame family of compact subsets of a completely metrizable space $X$ such that each set $K\in\K$ contains more than one point. For any open set $U\subset X$ there is a $\K$-tame decomposition $\DD$ of $X$ such that $\bigcup\DD^\circ$ is a dense subset of $U$.
\end{theorem}

We shall say that a topological space $X$ is {\em strongly locally homogeneous} if the family of singletons $\big\{\{x\}\big\}_{x\in X}$ is tame. This happens if and only if this family has the local shift property. So, our definition of the strong local homogeneity agrees with the classical one introduced in \cite{Bennett}. It is easy to see that each connected strongly locally homogeneous space is {\em topologically homogeneous} in the sense that for any two points $x,y\in X$ there is a homeomorphism $h:X\to X$ with $h(x)=y$.

The main technical result of this paper in the following theorem on the density of liftable homeomorphisms between decomposition spaces.

\begin{theorem}\label{main2}  For any tame family $\K$ of compact subsets of a strongly locally homogeneous completely metrizable space $X$ and any dense $\K$-tame decompositions $\A,\mathcal B$ of $X$, the set of $(\A,\BB)$-liftable homeomorphisms is dense in the homeomorphism space $\HH(\A,\BB)$.
\end{theorem}

The proof of this theorem will be presented in Section~\ref{s:main2} after long preparatory work in Sections~\ref{s:dense}--\ref{s:dense-Ktame}. Now we apply this theorem to prove the following corollary.

\begin{corollary}\label{c2.6} For any tame family $\K$ of compact subsets of a strongly locally homogeneous completely metrizable space $X$, any two dense $\K$-tame decompositions $\A,\mathcal B$ of $X$ are topologically equivalent. Moreover, for any open cover $\U$ of $X$ there is a homeomorphism $\Phi:X\to X$ such that $\Phi(\A)=\mathcal B$ and $(\Phi,\id_X)\prec\W$, where $$\W=\{\St(A,\U)\cup\St(B,\U):A\in\A,\;B\in\BB,\;\;\St(A,\U)\cap\St(\BB,\U)\ne\emptyset\}.$$
\end{corollary}

\begin{proof} Fix an open cover $\U$ of $X$. For every set $A\in\A$ consider its open neighborhood $\St(A,\U)=\{U\in\U:A\cap U\ne\emptyset\}$. Since the quotient map $q_\A:X\to\A=X/\A$ is closed, the set $O(A)=\A\setminus q_\A\big(X\setminus\St(A,\U)\big)$ is an open neighborhood of the point $A=q_\A(A)\in\A$ in the decomposition space $\A=X/\A$. By Lemma~\ref{l1}, the decomposition space $\A=X/\A$ is metrizable and hence paracompact. Consequently, we can find an open cover $\U_\A$ of $\A$ such that $\St(\U_\A)\prec\{O(A):A\in\A\}$. By analogy, choose an open cover $\U_\BB$ of the decomposition space $\BB$ such that $\St(\U_\BB)\prec\{O(B):B\in\BB\}$ where $O(B)=\BB\setminus q_\BB\big(X\setminus\St(B,\U)\big)$ for each $B\in\BB$.

By Definition~\ref{d:d-Ktame} and Theorem~\ref{t:shrink}, the quotient maps  $q_\A:X\to\A$ and $q_\BB:X\to\BB$ are near homeomorphism. Consequently, we can find homeomorphisms $h_\A:X\to\A$ and $h_\BB:X\to\BB$ such that $(h_\A,q_\A)\prec\U_\A$ and $(h_\BB,q_\BB)\prec\U_\BB$. Applying Theorem~\ref{main2}, find an $(\A,\BB)$-liftable homeomorphism $\varphi:\A\to\BB$ such that $(\varphi,h_B\circ h_A^{-1})\prec\U_B$. The $(\A,\BB)$-liftability of $\varphi$ yields a homeomorphism $\Phi:X\to X$ such that $q_\BB\circ\Phi=\varphi\circ q_\A$. The latter equality implies that $$\big\{\Phi(A):A\in\A\big\}=\big\{q_\BB^{-1}\circ\varphi\circ q_\A(A):A\in\A\big\}=\big\{q_\BB^{-1}\circ \varphi(\{A\}):A\in\A\big\}=\big\{q_\BB^{-1}(\{B\}):B\in\BB\big\}=\BB.$$

To show that $(\Phi,\id_X)\prec\W$, take any point $x\in X$ and consider the point $y=h_\A^{-1}\circ q_\A(x)\in X$. Since $(h_\A,q_\A)\prec\U_A$, there are a set $U\in\U_A$ and a set $A\in\A$ such that $\{q_\A(x),q_\A(y)\}=\{h_\A(y),q_\A(y))\subset U\subset O(A)$. Then $\{x,y\}\subset q^{-1}_\A(O(A)) \subset\St(A,\U)$.

The choice of the homeomorphism $h_B:X\to\BB$ guarantees that $\{h_\BB(y),q_\BB(y)\}\subset U$ for some set $U\in\U_\BB$. Since $(\varphi,h_\BB\circ h^{-1}_\A)\prec\U_\BB$, we conclude that $\{\varphi\circ q_\A(x),h_\BB\circ h^{-1}_\A\circ q_\A(x)\}\subset U'$ for some set $U'\in\U_\BB$. Then $h_\BB(y)=h_\BB\circ h_\A^{-1}\circ q_\A(x)\in U\cap U'$ and hence $\{\varphi\circ q_\A(x),q_\BB(y)\}\subset U\cup U'\subset O(B)$ for some set $B\in\BB$. The definition of the set $O(B)$ implies that
$$\{\Phi(x),y\}\subset q_\BB^{-1}\circ\varphi\circ q_\A(x)\cup q_\BB^{-1}\circ q_\BB(y)\subset q^{-1}_\BB(O(B))\subset \St(B,\U).$$

Consequently, $y\in\St(A,\U)\cap\St(B,\U)$ and $\{x,\Phi(x)\}\subset\St(A,\U)\cup\St(B,\U)\in \W$.
\end{proof}

\section{Approximating strong near homeomorphisms by homeomorphisms}

In this section we prove an auxiliary result on the approximation of strong near homeomorphisms by homeomorphisms. This result will be used in the proof of Theorem~\ref{t5}.

\begin{lemma}\label{l:approx} Let $\DD$ be a vanishing decompositions of a metrizable space $X$, $U\subset \DD$ be an open neighborhood of the non-degeneracy part $\DD^\circ$ in the decomposition space $\DD=X/\DD$, and $V=q_\DD^{-1}(U)\subset X$. Then there is an open cover $\U$ of $U$ such that for any homeomorphism $h:V\to U$ with $(h,q_\DD|V)\prec\U$ the map $\bar h:X\to\DD$ defined
$$\bar h(x)=\begin{cases}
h(x)&\mbox{if $x\in V$}\\
\{x\}&\mbox{otherwise}
\end{cases}
$$
is a homeomorphism of $X$ onto the decomposition space $\DD=X/\DD$.
\end{lemma}

\begin{proof} Fix a metric $d$ generating the topology of the space $X$ and let $\V$ be an open cover of the set $V=q_\DD^{-1}(U)$ such that $\St(\V)\prec\{O_d(v,d(v,X\setminus V)/2):v\in V\}$.

\begin{claim}\label{cl:ed} For each point $x_0\in X\setminus V$, and each $\epsilon>0$ there is a positive $\delta\le\epsilon$ such that for each $D\in\DD$, if $x\notin D$ and $\St(D,\V)\cap O_d(x_0,\delta)\ne\emptyset$, then  $\St(D,\V)\subset O_d(x_0,\epsilon)$.
\end{claim}

\begin{proof} Consider the open cover $\{O_d(x_0,\epsilon/2),X\setminus \bar O_d(x_0,\epsilon/4)\}$ of the space $X$. Since the decomposition $\DD$ is vanishing, the family $\DD'=\{D\in\DD:x_0\notin D,\;D\not\subset O_d(x_0,\epsilon/2),\;D\not\subset X\setminus \bar O_d(x_0,\epsilon/4)\}$
is discrete in $X$ and hence has closed union $\bigcup\DD'$, which does not contain the point $x_0$. Then we can find a positive $\delta<\epsilon/6$ such that $O_d(x_0,3\delta/2)\cap\bigcup\DD'=\emptyset$. Assume now  that $\St(D,\V)\cap O_d(x_0,\delta)\ne\emptyset$ for some set $D\in\DD$ with $x_0\notin D$. Pick any point $x\in \St(D,\V)\cap O_d(x_0,\delta)$ and find a point $z\in D\cap\St(x,\V)\subset O_d(x,d(x,X\setminus V)/2)$. Since $$d(z,x_0)\le d(z,x)+d(x,x_0)<\frac12d(x,X\setminus V)+d(x,x_0)\le \frac32d(x,x_0)<\frac32\delta<\frac14\epsilon,$$
the set $D$ meets the ball $O_d(x_0,3\delta/2)$ and hence does not belong to the family $\DD'$. Taking into account that the set $D\notin\DD'$ meets the ball $O_d(x_0,\epsilon/4)$, we conclude that $D\subset O_d(x_0,\epsilon/2)$. Given any point $y\in\St(D,\U)$, observe that $\emptyset\ne D\cap\St(y,\V)\subset O_d(x_0,\epsilon/2)\cap O_d(y,O_d(y,X\setminus V)/2)$ and hence $$d(x_0,y)< \frac12\epsilon+\frac12d(y,X\setminus V)\le\frac12\epsilon+\frac12d(y,x_0),$$which implies that $d(y,x_0)<\epsilon$ and $\St(D,\V)\subset O_d(x_0,\epsilon)$.
\end{proof}

The decomposition $\DD$ induces the decomposition $\DD_V=\{D\in\DD: D\subset V\}$ of the space $V$.
By Lemma~\ref{l2}, the vanishing decomposition $\DD$ is upper semicontinuous and hence the quotient map $q_\DD:X\to \DD$ is closed. Consequently, for every set $D\in\DD_V\subset\DD$ the set $F=q_\DD(X\setminus\St(D,\V))$ is closed in $\DD$ and the set $O(D)=U\setminus F$ is an open neighborhood of the point $D\in\DD$ in the decomposition space $\BB$. Since $\bigcup\DD_V=V$, the family $\U=\{O(D):D\in\DD_V\}$ is an open cover of the open subspace $U=q_\DD(V)$ of the decomposition space $\DD=X/\DD$. We claim that the open cover $\U$ has the property required in Lemma~\ref{l:approx}.

Let $h:V\to U$ be a homeomorphism with $(h,q_\DD|V)\prec\U$ and $\bar h:X\to\DD$ be the extension of $h$ such that $\bar h(x)=\{x\}$ for all $x\in X\setminus V$. It is clear that the map $\bar h$ is bijective. Since $\bar h|V=h|V$, the map $\bar h$ is open and continuous at each point $x_0\in V$. So, it remains to prove the continuity and the openness of $\bar h$ at each point $x_0\in X\setminus V$.

To prove the continuity of $\bar h$ at $x_0$, take any neighborhood $O(\{x_0\})\subset\DD$ of the image $h(x_0)=\{x_0\}$ of $x_0$ in the decomposition space $\DD$. By the continuity of the quotient map $q_\DD$ the preimage $O(x_0)=q_\DD^{-1}\big(O(\{x_0\})\big)$ of this neighborhood is a $\DD$-saturated open neighborhood of the point $x_0$ in $X$. Find a positive $\epsilon$ such that $O_d(x_0,\epsilon)\subset O(x_0)$. By Claim~\ref{cl:ed}, there a positive number $\delta\le\e$ such that for each set $D\in\DD_V$ with $O_d(x_0,\delta)\cap\St(D,\V)\ne\emptyset$, we get $\St(D,\V)\subset O_d(x_0,\epsilon)$.

We claim that $\bar h(O_d(x_0,\delta))\subset O(\{x_0\})$. Pick any point $x\in O_d(x_0,\delta)$. If $x\notin V$, then $x\in O_d(x_0,\delta)\subset O_d(x_0,\epsilon)\subset O(x_0)=q_\BB^{-1}(O(\{x_0\})$ and hence $\bar h(x)=\{x\}=q_\BB(x)\in O(\{x_0\})$. So, we assume that $x\in V$. In this case $\bar h(x)=h(x)$ and $(h(x),q_\BB(x))\subset O(D)\in\U$ for some set $D\in \DD_V$. Then $\{x\}\cup q_\BB^{-1}(h(x))\subset q_\BB^{-1}(O(D))\subset\St(D,\V)$. Since $x\in \St(D,\V)\cap O_d(x_0,\delta)$, the choice of the number $\delta$ guarantees that $q_\BB^{-1}(h(x))\subset\St(x_0,\V)\subset O_d(x_0,\e)\subset O(x_0)$ and hence $h(x)\in q_\BB(O(x_0))=O(\{x_0\})$.
So, the map $\bar h:X\to\DD$ is continuous at $x_0$.

Next, we show that the map $\bar h$ is open at $x_0$. Given any $\epsilon>0$, we should find an open neighborhood $U(\{x_0\})\subset\DD$ of the point $\{x_0\}=\bar h(x_0)=q_\DD(x_0)$ such that $U(\{x_0\})\subset h(O_d(x_0,\epsilon))$. By Claim~\ref{cl:ed}, there exists a positive number $\delta\le \epsilon$ such that for each set $D\in\DD_V$ with $\St(D,\V)\cap O_d(x_0,\delta)\ne\emptyset$, we get $\St(D,\V)\subset O_d(x_0,\epsilon)$. Since the decomposition $\DD$ is upper semicontinuous, for the closed subset $C=X\setminus O_d(x_0,\delta)$ of $X$ its $\DD$-saturation $\St(C,\DD)$ is  closed in $X$. Then $U(x_0)=X\setminus \St(C,\DD)\subset O_d(x_0,\delta)$ is a $\DD$-saturated open neighborhood of $x_0$ in $X$ and its image $U(\{x_0\})=q_\DD(U(x_0))$ is an open neighborhood of the point $\{x_0\}$ in the decomposition space $\DD$.  We claim that $U(\{x_0\})\subset \bar h(O_d(x_0,\e))$. Take any point $y\in U(\{x_0\})$ and consider its preimage $x=\bar h^{-1}(y)\in X$. If $x\notin V$, then $y=\bar h(x)=q_\DD(x)=\{x\}$ and hence $x\in q_\DD^{-1}(y)\subset q_\DD^{-1}\big(U(\{x_0\})\big)=U(x_0)\subset O_d(x_0,\delta)\subset O_d(x_0,\epsilon)$.
So, we assume that $y\in U$. In this case $y=\bar h(x)=h(x)$. Since $(h,q_\DD|V)\prec\U$, there is a set $D\in\DD_V$ such that $\{q_\DD(x),y\}=\{q_\DD(x),h(x)\}\subset O(D)\in\U$ and thus $q_\DD^{-1}(y)\subset q_\DD^{-1}(\{q_\DD(x),y\})\subset q_\DD^{-1}\big(O(D)\big)\subset \St(D,\V)$ by the choice of the neighborhood $O(D)$. Taking into account that $q_\DD^{-1}(y)\subset U(x_0)\subset O_d(x_0,\delta)$, we see that the $\V$-star $\St(D,\V)$ of $D$ meets the $\delta$-ball $O_d(x_0,\delta)$ and hence is contained in the $\e$-ball $O_d(x_0,\epsilon)$ by the choice of $\delta$. Then $x\in q_\DD^{-1}\big(q_\DD(x)\big)\subset \St(D,\V)\subset O_d(x_0,\epsilon)$ and $y=\bar h(x)\in \bar h\big(O_d(x_0,\epsilon)\big)$.
\end{proof}

\section{Topological equivalence of dense $\sigma$-discrete subsets of strongly locally homogeneous spaces}\label{s:dense}

In this section we establish one important property of strongly locally homogeneous completely metrizable spaces, which will be used several times in the proof of Theorem~\ref{main2} and \ref{t5}.

Let us recall that a topological space $X$ is called {\em strongly locally homogeneous} if for each point $x\in X$ and an open neighborhood $O_x\subset X$ of $x$ there is an open neighborhood $U_x\subset O_x$ of $x$ such that for any point $y\in U_x$ there is a homeomorphism $h:X\to X$ such that $h(x)=y$ and $h|X\setminus O_x=\id$.

A subset $D$ of a topological space $X$ is called {\em $\sigma$-discrete} if $D$ can be written as a countable union $D=\bigcup_{n\in\w}D_n$ of closed discrete subsets of $X$.

The following theorem generalizes a result of Bennett \cite{Bennett} on the topological equivalence of any  countable dense subsets in a strongly locally homogeneous Polish space.

\begin{theorem}\label{t:dense} If $X$ is a strongly locally homogeneous completely metrizable space, then for any open cover $\U$ of an open subspace $U\subset X$, and any dense $\sigma$-discrete subspaces $A,B\subset U$ there is a homeomorphism $h:X\to X$ such that $h(A)=B$ and $(h,\id)\prec\U$.
\end{theorem}

\begin{proof} Since the strong local homogeneity is inherited by open subspaces, we lose no generality assuming that $U=X$. Using a standard technique of Tukey (cf. \cite[5.4.H]{En}), we can choose a complete metric $d$ generating the topology of $X$ and such that the cover $\{O_d(x,1):x\in X\}$ of $X$ by closed 1-balls refines the cover $\U$.

Given dense $\sigma$-discrete subsets $A,B$ in $U=X$, choose a (not necessarily continuous) function $\delta:A\cup B\to (0,1]$ such that for each $\epsilon>0$ the set $\{x\in A\cup B:\delta(x)>\epsilon\}$ is closed and discrete in $X$.

We shall construct inductively a sequence of homeomorphisms $(h_n:X\to X)_{n\in\w}$ and two sequences $(A_n)_{n\in\w}$ and $(B_n)_{n\in\w}$ of closed discrete subsets of $X$ such that for every $n\in\w$ the following conditions will be satisfied:
\begin{enumerate}
\item $A_{n-1}\cup\{a\in A:\delta(a)\ge 2^{-n}\}\subset A_n\subset A$;
\item $B_{n-1}\cup\{b\in B:\delta(b)\ge 2^{-n}\}\subset B_n\subset B$;
\item $h_n(A_n\setminus A_{n-1})=B_n\setminus B_{n-1}$;
\item $h_n|A_{n-1}=h_{n-1}|A_{n-1}$, and
\item $d(h_n,h_{n-1})\le 2^{-n-1}$ and $d(h_n^{-1},h_{n-1}^{-1})\le 2^{-n-1}$.
\end{enumerate}

We start the inductive construction by letting $A_0=B_0=\emptyset$ and $h_0=\id_X$. Assume that for some $n\in\IN$ subsets $A_i,B_i$ and homeomorphisms $h_i$ have been constructed for all $i<n$.
The inductive assumptions (3) and (4) imply that $h_{n-1}(A_{n-1})=B_{n-1}$.

Consider the subsets $\tilde A_n=\{a\in A\setminus A_{n-1}:\delta(a)\ge 2^{-n}\}$ and $\tilde B_n=\{b\in B\setminus B_{n-1}:\delta(b)\ge 2^{-n}\}$. By the choice of the function $\delta$, these sets are closed and discrete in $X$. Then the sets $B_n'=h_{n-1}(\tilde A_n)\setminus \tilde B_n$ and $A_n'=h_{n-1}^{-1}(\tilde B_n)\setminus\tilde A_n$ also are closed and discrete in $X$. It follows that $h_{n-1}(A'_n)\cap B_n'=\emptyset$. By normality of the space $X$, the closed sets $A_n',B_n'$ have  open neighborhoods $O(A'_n),O(B'_n)\subset X$ such that $h_{n-1}(\bar O(A'_n))\cap \bar O(B'_n)=\emptyset$, where $\bar O(A_n')$ and $\bar O(B_n')$ are the closures of these neighborhoods in $X$.
Moreover, we can assume that $\bar O(A_n')\cap(A_{n-1}\cup \tilde A_n)=\emptyset$ and
$\bar O(B_n')\cap(B_{n-1}\cup \tilde B_n)=\emptyset$.

For each point $b\in B_n'$ choose a neighborhood $V_b\subset O(B_n')$ such that $\diam(V_b)<2^{-n-1}$ and $\diam(h^{-1}_{n-1}(V_b))<2^{-n-1}$. Since the set $B_n'$ is closed and discrete in the collectionwise normal space $X$, we can assume that the family $(V_b)_{b\in B_n'}$ is discrete in $X$. Since the space $X$ is strongly locally homogeneous, each point $b\in B_n'$ has a neighborhood $W_b\subset V_b$ such that for each point $b'\in W_b$ there is a homeomorphism $\beta_b:X\to X$ such that $\beta_b(b)=b'$ and $\beta_b|X\setminus V_b=\id$. Since the subset $B\subset X$ is dense, we can choose a point $b'\in B\cap W_b$ and find a homeomorphism such that $\beta_b(b)=b'$ and $\beta_b|X\setminus V_b=\id$. The homeomorphisms $\beta_b$, $b\in B_n'$, produce a single homeomorphism
$\beta:X\to X$ defined by the formula
$$
\beta(x)=\begin{cases}\beta_b(x)&\mbox{if $x\in V_b$ for some $b\in B_n'$,}\\
x&\mbox{otherwise.}
\end{cases}
$$
It is easy to see that the homeomorphism $\beta:X\to X$ has the following properties:
\begin{itemize}
\item $\beta(B_n')\subset B$,
\item $\beta|X\setminus O(B'_n)=\id$,
\item $d(\beta\circ h_{n-1},h_{n-1})\le 2^{-n-1}$, and
\item $d(h_{n-1}^{-1}\circ \beta^{-1},h_{n-1}^{-1})\le 2^{-n-1}$.
\end{itemize}
Let us prove the latter inequality. Given any point $x\in X$, we need to check that $d(h_{n-1}^{-1}\circ\beta^{-1}(x),h_{n-1}^{-1}(x))\le2^{-n-1}$. If $x\notin\bigcup_{b\in b_n'}V_b$, then $\beta(x)=x=\beta^{-1}(x)$ and hence $d(h_{n-1}^{-1}\circ\beta^{-1}(x),h_{n-1}^{-1}(x))=0\le 2^{-n-1}$.
So, we assume that $x\in V_b$ for some $b\in B'_n$. Then the point $y=\beta^{-1}(x)$ also belongs to $V_b$ and hence $$d\big(h_{n-1}^{-1}\circ \beta^{-1}(x),h_{n-1}^{-1}(x)\big)\le\diam \big(h_{n-1}^{-1}(V_b)\big)\le 2^{-n-1}$$by the choice of the neighborhood $V_b$.

By analogy we can construct a homeomorphism $\alpha:X\to X$ such that
\begin{itemize}
\item $\alpha(A_n')\subset A$,
\item $\alpha|X\setminus O(A'_n)=\id$,
\item $d(\alpha\circ h^{-1}_{n-1},h^{-1}_{n-1})\le 2^{-n-1}$, and
\item $d(h_{n-1}\circ \alpha^{-1},h_{h-1})\le 2^{-n-1}$.
\end{itemize}

Let $A_n=A_{n-1}\cup\tilde A_n\cup \alpha(A_n')$ and $B_n=B_{n-1}\cup\tilde B_n\cup \beta(B_n')$.
Now consider the homeomorphism $h_n:X\to X$ defined by the formula
$$h_n(x)=\begin{cases} \beta\circ h_{n-1}(x)&\mbox{if $x\in h_{n-1}^{-1}(O(B'_n))$},\\
h_{n-1}\circ\alpha^{-1}(x)&\mbox{if $x\in O(A_n')$}\\
h_{n-1}(x)&\mbox{otherwise}.
\end{cases}
$$
The choice of the neighborhoods $O(A_n')$ and $O(B'_n)$ guarantees that $h_n$ is a well-defined homeomorphism that satisfies the conditions (1)--(5) of the inductive construction. This completes the inductive step.

The condition (5) of the inductive construction imply that the limit map $h=\lim_{n\to\infty}h_n$ is a homeomorphism of $X$ such that $$d(h,\id)\le \sum_{n=1}^\infty d(h_n,h_{n-1})\le\sum_{n=1}2^{-n-1}=1$$and hence $(h,\id)\prec\U$ by the choice of the metric $d$.

The conditions (3) and (4) of the inductive construction imply that $h|A_n=h_n|A_n$ and $h_n(A_n)=B_n$ for all $n\in\w$. Taking into account that  $A=\bigcup_{n\in\w}A_n$ and
 $B=\bigcup_{n\in\w}B_n$, we conclude that $h(A)=B$.
\end{proof}

\section{Topological equivalence of discrete $\K$-tame decompositions}

In this section we shall prove a discrete version of Theorem~\ref{main2}. We recall that a decomposition $\DD$ of a topological space $X$ is called {\em discrete} if its non-degeneracy part $\DD^\circ=\{D\in\DD:|D|>1\}$ is closed and discrete in the decomposition space $\DD=X/\DD$.

The following fact easily follows from the definitions.

\begin{lemma}\label{l:dshrink} A discrete decomposition $\DD$ of a regular topological space is strongly shrinkable if and only if each set $D\in\DD$ is locally shrinkable in $X$.
\end{lemma}

For two decompositions $\A,\BB$ of a topological space $X$ we shall denote by $\HH^\circ(\A^\circ,\BB^\circ)$  the space of all  homeomorphisms $h:(\A,\A^\circ)\to(\BB,\BB^\circ)$ of the pairs $(\A,\A^\circ)$ and $(\BB,\BB^\circ)$, endowed with the strong limitation topology, whose neighborhood base at a homeomorphism $h\in\HH^\circ(\A^\circ,\BB^\circ)$ consists of the sets
$$N(h,\U)=\{g\in\HH^\circ(\A^\circ,\BB^\circ):(f,g)\prec\U\}$$where $\U$ runs over all covers of the non-degeneracy part $\BB^\circ$ by open subsets of the decomposition space $\BB$.

\begin{theorem}\label{t4} Let $\K$ be a tame family $\K$ of compact subsets of a strongly locally homogeneous completely metrizable space $X$. Then for any discrete decompositions $\A,\BB\subset\K\cup\big\{\{x\}:x\in X\big\}$ of $X$, the set of $(\A,\BB)$-liftable homeomorphisms is dense in the homeomorphism space $\HH^\circ(\A^\circ,\BB^\circ)$.
\end{theorem}

\begin{proof} Given a homeomorphism of pairs $f:(\A,\A^\circ)\to(\BB,\BB^\circ)$ and a cover $\W$ of the non-degeneracy part $\BB^\circ$ by open subsets of $\BB$, we need to construct a $(\A,\BB)$-liftable homeomorphism $\varphi:\A\to\BB$ such that $(\varphi,f)\prec\W$.

Since the decomposition $\BB$ is discrete, its non-degeneracy part $\BB^\circ$ is closed and discrete in the decomposition space $\BB=X/\BB$. Then we can choose for every point $b\in \BB^\circ$  an open neighborhood $W_b\subset\BB$ of $b$, which lies in some set of the cover $\W$. Moreover, since the set $\BB^\circ$ is closed and discrete in the metrizable (and collectionwise normal) space $\BB$, we can additionally assume that the indexed family $\{W_b:b\in\BB^\circ\}$ is discrete in $\BB$.

By Definition~\ref{d:K-tame} and Lemma~\ref{l:dshrink}, the discrete decomposition $\BB$ is strongly shrinkable and by Theorem~\ref{t:shrink}, the quotient map $q_\BB:X\to \BB$ is a strong near homeomorphism, which implies that the decomposition space $\BB$ is homeomorphic to $X$. Then  $\K(\BB)=\{h(K):K\in\K,\;\;h\in\HH(\BB,X)\}$ is a tame family of compact subsets in the space $\BB$. This family has the local shift property, which implies that each point $b\in\BB^\circ$ has a neighborhood $U_b\subset W_b$ such that for any compact subsets $K,K'\in \K(\BB)$ of $U_b$ there is a homeomorphism $h_b:\BB\to\BB$ such that $h_b(K)=K'$ and $h_b|\BB\setminus W_b=\id$. Let $U=\bigcup_{b\in\BB^\circ}U_b$.

Since the quotient map $q_\BB:X\to\BB$ is a strong near homeomorphism, there is a homeomorphism $\beta:X\to \BB$ such that $\beta(q_\BB^{-1}(U_b))=U_b$ for every $b\in\BB^\circ$ and
$\beta(x)=\{x\}$ for each $x\in X\setminus q_\BB^{-1}(U)$.

By analogy we shall define a homeomorphism $\alpha:X\to\A$. Namely, for every point $a\in\A^\circ$ consider the open neighborhood $V_a=f^{-1}(U_{f(a)})$ of $a$ in the decomposition space $\A$, and put $V=\bigcup_{a\in\A^\circ}V_a=f^{-1}(U)$. Since the decomposition $\A$ is strongly shrinkable, the quotient map $q_\A:X\to\A$ is a strong near homeomorphism, which allows us to find a homeomorphism $\alpha:X\to \A$ such that $\alpha(q_\A^{-1}(V_a))=V_a$, for every $a\in\A^\circ$, and
$\alpha(x)=\{x\}$ for each $x\in X\setminus q_\A^{-1}(V)$.

For every $b\in\BB^\circ$, consider the point $a=f^{-1}(b)\in\A^\circ$ and the compact subsets $K=\beta(b)$ and $K'=f\circ\alpha(a)$ of $U_b$, which belong to the family $\K(\BB)$. By the choice of the neighborhood $U_b$, there exists a homeomorphism $h_b:\BB\to\BB$ such that $h_b(K')=K$ and $h_b|\BB\setminus W_b=\id$. The homeomorphisms $h_b$, $b\in\BB^\circ$, yield a single homeomorphism $h:\BB\to\BB$ defined by
$$h(y)=\begin{cases}h_b(y)&\mbox{if $y\in W_b$ for some $b\in\BB^\circ$};\\
y&\mbox{otherwise}.
\end{cases}
$$
Consider the homeomorphism $\Phi=\beta^{-1}\circ h\circ f\circ\alpha:X\to X$. The definition of the homeomorphism $h$ implies that for every compact set $a\in\A^\circ$ of $X$ and its image $b=f(a)\in\BB^\circ$ we get $$\Phi(a)=\beta^{-1}\circ h\circ f\circ\alpha(a)=
\beta^{-1}\circ h_b(f\circ\alpha(a))=\beta^{-1}\circ \beta(b)=b.$$
This means that the homeomorphism $\Phi$ is $(\A,\BB)$-factorizable and hence there is a homeomorphism $\varphi:\A\to\BB$ such that $q_\BB\circ\Phi=\varphi\circ q_\A$. The choice of the neighborhoods $W_b$, $b\in\BB^\circ$, guarantees that the  $(\A,\BB)$-liftable homeomorphism $\varphi:\A\to\BB$ is $\W$-near to the homeomorphism $f$.
\end{proof}

\section{Topological equivalence of dense $\K$-tame decompositions}\label{s:dense-Ktame}

In the proof of Theorem~\ref{t5} below we shall widely use multivalued maps; see \cite{RS}. By a multivalued map $\Phi:X\setmap Y$ between sets $X$ and $Y$ we understand any subset $\Phi\subset X\times Y$ of their Cartesian product.
This subset $\Phi$ can be thought of as a multivalued function $\Phi:X\setmap Y$ which assigns to each point $x\in X$ the subset $\Phi(x)=\{y\in Y:(x,y)\in\Phi\}$ of $Y$ and to each subset $A\subset X$ the subset $\Phi(A)=\bigcup_{a\in A}\Phi(a)$ of $Y$.
Usual functions $f:X\to Y$, identified with their graphs $\{(x,f(x)):x\in X\}$, become multivalued (more precisely, singlevalued) functions.

For two multivalued functions $\Psi:X\setmap Y$ and $\Psi:Y\setmap Z$ their composition $\Psi\circ\Phi:X\setmap Z$ is defined as the multivalued function assigning to each point $x\in X$ the subset $\Psi(\Phi(x))$ of $Z$. The inverse $\Phi^{-1}$ of a multivalued function $\Phi:X\setmap Y$ is the multivalued function $\Phi^{-1}=\{(y,x):(x,y)\in\Phi\}\subset Y\times X$, assigning to each point $y\in Y$ the subset $\Phi^{-1}(y)=\{x\in X:y\in\Phi(x)\}$.

\begin{theorem}\label{t5}  For any tame family $\K$ of compact subsets of a strongly locally homogeneous completely metrizable space $X$, and any dense $\K$-tame decompositions $\A,\BB$ of $X$, the set of $(\A,\BB)$-liftable homeomorphisms is dense in the homeomorphism space $\HH^\circ(\A^\circ,\BB^\circ)$.
\end{theorem}

\begin{proof} Given a homeomorphism of pairs $\varphi_0:(\A,\A^\circ)\to(\BB,\BB^\circ)$ and a cover $\W$ of the non-degeneracy part $\BB^\circ$ by open subsets of $\BB$, we need to construct a $(\A,\BB)$-liftable homeomorphism $\varphi:\A\to\BB$ such that $(\varphi,\varphi_0)\prec\W$.

Fix a complete metric $d$ that generates the topology of the completely metrizable space $X$. Replacing $d$ by $\min\{d,1\}$, if necessary, we can assume that $\diam(X)\le 1$. Also fix a metric $\rho\le1$ degenerating the topology of the decomposition space $\BB=X/\BB$ (which is metrizable by Lemma~\ref{l1}).
Choose a continuous function $\e:\BB\to[0,1]$ such that $\e^{-1}(0)=\BB\setminus\bigcup\W$ and for each point $b\in\bigcup\W$ the closed $\e(b)$-ball $\bar O_\rho(b,\e(b))=\{y\in \BB:\rho(y,b)\le\e(b)\}$ is contained in some element of the cover $\W$. Then each map $\varphi:\A\to\BB$ with $\rho(\varphi,\varphi_0)\le\e\circ \varphi_0$ is $\W$-near to the map $\varphi_0$.
So, it suffices to construct a $(\A,\BB)$-liftable homeomorphism $\varphi:\A\to\BB$ such that $\rho(\varphi(a),\varphi_0(a))\le\e\circ\varphi_0(a)$ for every $a\in\A$.

%All the metric notions (like balls, diameters, etc.) in the spaces $X$ and $\BB$ will be referred to the metrics $d$ and $\rho$, respectively.

To find such a homeomorphism $\varphi$, we shall construct inductively two sequences $(\A_n)_{n\in\w}$ and $(\BB_n)_{n\in\w}$ of decompositions of the space $X$, and two sequences of homeomorphisms $(h_n:\A_n\to\BB_n)_{n\in\w}$, $(\varphi_n:\A\to\BB)_{n\in\w}$ between the corresponding decomposition spaces such that for the multivalued functions
$\Phi_n=q_{\BB_n}^{-1}\circ h_n\circ q_{\A_n}:X\setmap X$, $n\in\w$, the
the following conditions are satisfied for every $n\ge 1$:
\begin{enumerate}
\item[$(1_n)$] $\A_{n}^\circ\subset \A^\circ_{n-1}\subset\A^\circ$ and $\BB^\circ_{n}\subset\BB^\circ_{n-1}\subset\BB^\circ$,
\item[$(2_n)$] the families $\A^\circ_{n-1}\setminus \A^\circ_{n}$ and $\BB^\circ_{n-1}\setminus \BB^\circ_{n}$ are discrete in $X$ and\newline contain the families $\{A\in\A_n:\diam(A)\ge 2^{-n+1}\}$ and $\{B\in\BB_n:\diam(B)\ge 2^{-n+1}\}$, respectively;
\item[$(3_n)$] $q_\BB^{\BB_{n}}\circ h_{n}=\varphi_{n}\circ q_\A^{\A_{n}}$;
\item[$(4_n)$] $\rho(\varphi_{n},\varphi_{n-1})\le 2^{-n}\cdot \e\circ\varphi_0$;
\item[$(5_n)$] $\varphi_{n}|\A_0^\circ\setminus\A_{n}^\circ= \varphi_{n-1}|\A_0^\circ\setminus\A_{n}^\circ$;
\item[$(6_n)$] $\varphi_{n}(\A_{n}^\circ)=\BB_{n}^\circ$ and $\varphi_{n}(\A_{n-1}^\circ\setminus\A_{n}^\circ)=\BB_{n-1}^\circ\setminus\BB_{n}^\circ$;
\item[$(7_n)$] $\Phi_{n}|\bigcup(\A_0^\circ\setminus\A_{n-1}^\circ)=
    \Phi_{n-1}|\bigcup(\A_0^\circ\setminus\A^\circ_{n-1})$;
%\item $\diam\big(\Phi_{n+1}(x)\big)\le 2^{-n-1}$ and $\diam \Phi_{n+1}^{-1}(x)\le 2^{-n-1}$ for all $x\in X$;
\item[$(8_n)$] $\diam\big(\Phi_{n}(x)\cup\Phi_{n-1}(x)\big)<2^{-n+2}$ and $\diam\big(\Phi_{n}^{-1}(x)\cup\Phi_{n-1}^{-1}(x)\big)<2^{-n+2}$ for all $x\in X$.
\end{enumerate}

So, for every $n\in\w$ we shall inductively construct decompositions $\A_n$, $\BB_n$, homeomorphisms $h_n:\A_n\to\BB_n$, $\varphi_n:\A\to\BB$, and a multivalued function $\Phi_n:X\setmap X$ making the following diagram commutative
$$\xymatrix{
X\ar[r]^{\Phi_n}\ar[d]_{q_{\A_n}}&X\ar[d]^{q_{\BB_n}}\\
\A_n\ar[r]^{h_n}\ar[d]_{q_\A^{\A_n}}&\BB_n\ar[d]^{q_{\BB_n}^{\BB_{n+1}}}\\
\A\ar[r]_{\varphi_n}&\BB
}
$$

We start the inductive constructing putting $\A_0=\A$, $\BB_0=\BB$, and  $h_0=\varphi_0$.
\smallskip

{\bf Inductive step.} Assume that for some $n\in\w$  decompositions $\A_i$, $\BB_i$, $i\le n$, and homeomorphisms $h_i:\A_i\to\BB_i$, $\varphi_i:\A\to\BB$, $i\le n$, satisfying the conditions $(1_i)$--$(8_i)$, $1\le i\le n$, have been constructed. We should construct  decompositions $\A_{n+1}$ and $\BB_{n+1}$ of $X$ and homeomorphisms $h_{n+1}:\A_{n+1}\to\BB_{n+1}$ and $\varphi_{n+1}:\A\to\BB$.

Consider the decomposition spaces $\A_{n}=X/\A_{n}$, $\BB_{n}=X/\BB_{n}$, and the corresponding quotient maps $q_{\A_{n}}:X\to \A_{n}$ and $q_{\BB_{n}}:X\to \BB_{n}$.

By the conditions $(2_k)$, $k\le n$, the family $\A_0^\circ\setminus\A_{n}^\circ$ is discrete in $X$. Consequently, its union $\bigcup(\A_0^\circ\setminus\A_{n}^\circ)$ is closed in $X$ and its projection $\bar A_n=q_{\A_{n}}\big(\bigcup(\A_0^\circ\setminus\A_{n}^\circ)\big)$ is closed in the decomposition space $\A_{n}=X/\A_{n}$. By the same reason, the set $\bar B_n=q_{\BB_{n}}\big(\textstyle{\bigcup}
(\BB_0^\circ\setminus\BB_{n}^\circ)\big)$ is closed in the decomposition space $\BB_{n}=X/\BB_{n}$.

The density of the decomposition $\A$ implies that the set
$\bigcup\A_0^\circ$ is dense in $X$ and consequently the set
$$\A^\circ_n=q_{\A_n}(\textstyle{\bigcup}\A^\circ_n)=q_{\A_n}\big(\textstyle{\bigcup}\A_0^\circ)\setminus\bar A_n$$
is dense in the open subspace $\A_n\setminus \bar A_n$ of the decomposition space $\A_n=X/\A_n$.
By the same reason, the set $$\BB^\circ_n=q_{\BB_n}(\textstyle{\bigcup}\BB^\circ_n)=q_{\BB_n}\big(\textstyle{\bigcup}\BB_0^\circ)\setminus\bar B_n$$
is dense in the open subspace $\BB_n\setminus \bar B_n$ of the decomposition space $\BB_n=X/\BB_n$.

Since the decomposition $\A$ is vanishing and $\A_n^\circ\subset\A_0^\circ=\A^\circ$, the decomposition $\A_n$ is vanishing too. Consequently, for each $\e>0$ the subfamily $\A^\circ_{n,\e}=\{A\in\A_n:\diam(A)\ge\e\}$ is discrete in $X$, which implies that the set $\A^\circ_{n,\e}=q_{\A_n}(\bigcup\A_{n,\e})$ is closed and discrete in $\A_n$. Since $\A_n^\circ=\bigcup_{k=1}^\infty \A^\circ_{n,2^{-k}}$, the subset $\A^\circ_n$ is $\sigma$-discrete in  $\A_n\setminus \bar A_n$.
By analogy we can show that the set $\BB^\circ_n$ is $\sigma$-discrete in $\BB_n\setminus \bar B_n$.

Now consider the homeomorphisms $h_n:\A_n\to\BB_n$, $\varphi_n:\A\to\BB$, and the induced multivalued function $\Phi_n=q_{\BB_n}^{-1}\circ h_n\circ q_{\A_n}:X\setmap X$. The inductive assumption $(3_n)$, $(5_n)$, and $(6_n)$ imply that
$h_n(\bar A_n)=\bar B_n$ and $h_n(\A^\circ_n)=\BB_n^\circ$.

Since the decomposition $\A$ is vanishing, the family  $\A^\circ_{n,2^{-n}}=\{A\in\A^\circ_n:\diam(A)\ge 2^{-n}\}$ is discrete in $X$ and its image  $\A^\circ_{n,2^{-n}}=q_{\A_n}(\bigcup\A^\circ_{n,2^{-n}})\subset \A^\circ_n$ is a closed discrete subset of the decomposition space $\A_n=X/\A_n$. By the same reason, the family  $\BB^\circ_{n,2^{-n}}=\{B\in\BB^\circ_n:\diam(B)\ge 2^{-n}\}$ is discrete in $X$ and is a closed discrete subset $B^\circ_{n,2^{-n}}=q_{\BB_n}(\bigcup\BB^\circ_{n,2^{-n}})\subset \BB^\circ_n$ of the decomposition space $\BB_n=X/\BB_n$.

The conditions $(3_n)$ and $(6_n)$ of the inductive construction imply that $h_n(\A^\circ_n)=\BB_n^\circ$. Consequently, the closed discrete subset $\A^\circ_{n,2^{-n}}\cup h_n^{-1}(\BB^\circ_{n,2^{-n}})$ of the decomposition space $\A_n=X/\A_n$ is a subset of $\A^\circ_n$. By the same reason, the closed discrete subset $\BB^\circ_{n,2^{-n}}\cup h_n(\A^\circ_{n,2^{-n}})$ of the decomposition space $\BB_n=X/\BB_n$ is a subset of $\BB_n^\circ$. So, we can consider the subfamilies
$$\A^\circ_{n+1}=\big\{q_{\A_n}^{-1}(y):y\in \A^\circ_n\setminus\big(\A^\circ_{n,2^{-n}}\cup h_n^{-1}(\BB^\circ_{n,2^{-n}})\big)\big\}=\A^\circ_n\setminus\big(\A^\circ_{n,2^{-n}}\cup h_n^{-1}(\BB^\circ_{n,2^{-n}})\big)\subset\A^\circ_n$$and
$$\BB^\circ_{n+1}=\big\{q_{\BB_n}^{-1}(y):y\in \BB^\circ_n\setminus(\BB^\circ_{n,2^{-n}}\cup h_n(\A^\circ_{n,2^{-n}})\big)\big\}=\BB^\circ_n\setminus\big(\BB^\circ_{n,2^{-n}}\cup h_n(\A^\circ_{n,2^{-n}})\big)\subset\BB^\circ_n.$$
These subfamilies $\A_{n+1}^\circ\subset\A^\circ_n$ and $\BB^\circ_{n+1}\subset\BB_n^\circ$ generate the decompositions
$$
\begin{aligned}
\A_{n+1}&=\A_{n+1}^\circ\cup\big\{\{x\}:x\in X\setminus\textstyle{\bigcup}\A_{n+1}^\circ\big\}\mbox{ \ and \ }\\ \BB_{n+1}&=\BB_{n+1}^\circ\cup\big\{\{x\}:x\in X\setminus\textstyle{\bigcup}\BB_{n+1}^\circ\big\},
\end{aligned}
$$
of the space $X$, satisfying the conditions $(1_{n+1})$ and $(2_{n+1})$ of the inductive construction.

For every numbers $k,m\in\w$ with $0\le k\le m\le n+1$ the conditions $(1_k)$, $k\le n+1$, guarantee that $\A_k^\circ\subset\A_m^\circ$ and hence $\A_k\prec\A_m$. So, there is a (unique) map $q_{\A_k}^{\A_{m}}:\A_{m}\to\A_k$ making the following triangle commutative:
$$\xymatrix{
&X\ar[rd]^{q_{\A_k}}\ar[ld]_{q_{\A_m}}&\\
\A_m\ar[rr]^{q_{\A_k}^{\A_m}}&&\A_k.
}$$
This map $q_{\A_k}^{\A_{m}}:\A_m\to\A_k$ determines a decomposition $$\A_k^{m}=\big\{(q_{\A_k}^{\A_{m}})^{-1}(y):y\in\A_{k}\big\}=\big\{q_{\A_{m}}(A):A\in\A\big\}$$of the space $\A_{n+1}$. The non-degeneracy part $$(\A_k^{m})^\circ=\big\{q_{\A_{m}}(A):A\in\A_k^\circ\setminus\A_{m}^\circ\}$$ of this decomposition is discrete in $\A_{m}$ by the conditions $(2_i)$, $k<i\le m$, of the inductive construction.

By analogy, for any $0\le k\le m\le n+1$ we can define the map $q_{\BB_k}^{\BB_m}:\BB_m\to\BB_k$ and the corresponding decomposition $\BB_k^m=\{(q_{\BB_k}^{\BB_m})^{-1}(y):y\in\BB_k\}=\{q_{\BB_m}(B):B\in\BB\}$ of the decomposition space $\BB_{m}$.

Now consider the diagram:
$$\xymatrix{
&
X\ar@<2pt>@{..>}[r]^{\Phi_{n+1}}\ar@<-2pt>[r]_{\Phi_n}\ar[d]_{q_{\A_{n+1}}}&X\ar[d]^{q_{\BB_{n+1}}}&
\\
&\A_{n+1}\ar@<2pt>@{..>}[r]^{h_{n+1}}\ar@<-2pt>@{..>}[r]_{\tilde h_{n+1}}\ar[d]_{q_{\A_n}^{\A_{n+1}}}&
\BB_{n+1}\ar[d]^{q_{\BB_n}^{\BB_{n+1}}}&\\
\A_{n}^\circ{\setminus}\A_{n+1}^\circ\ar[r]\ar[d]&
\A_{n}\ar@<2pt>@{..>}[r]^{\tilde h_{n}}\ar@<-2pt>[r]_{h_{n}}\ar[d]_{q_{\A_0}^{\A_{n}}}&
\BB_{n}\ar[d]^{q_{\BB_0}^{\BB_{n}}}&\ar[l]\BB_{n}^\circ{\setminus}\BB_{n+1}^\circ\ar[d]\\
\A^\circ\ar[r]&\A\ar@<2pt>[r]^{\varphi_n}\ar@<-2pt>@{..>}[r]_{\varphi_{n+1}}&\BB&\ar[l]\BB^\circ
}
$$
In this diagram the straight arrows denote the maps which are already defined while dotted arrows denote maps which will be constructed during the inductive step in the following way. First, using Theorem~\ref{t4} we approximate the homeomorphism $h_n$ by a $(\A_n^{n+1},\BB_n^{n+1})$ liftable homeomorphism $\tilde h_n$, which determines a homeomorphism $\tilde h_{n+1}:\A_{n+1}\to\BB_{n+1}$. Then using Theorem~\ref{t:dense} we approximate the homeomorphism $\tilde h_{n+1}$ by a $(\A^{n+1}_0,\BB^{n+1}_0)$-factorizable homeomorphism $h_{n+1}$ such that $h_n(\A_{n+1}^\circ)=\BB_{n+1}^\circ$. The homeomorphism $h_{n+1}$ determines a homeomorphism $\varphi_{n+1}:\A\to\BB$ and the multivalued function $\Phi_{n+1}=q^{-1}_{\BB_{n+1}}\circ h_{n+1}\circ q_{\A_{n+1}}:X\setmap X$, which will satisfy the inductive assumptions $(3_{n+1})$--$(8_{n+1})$.
Now we realize this strategy in details.

The homeomorphism $\tilde h_n$ will differ from the homeomorphism $h_n$ on a neighborhood $U'_n\subset\A_{n}$ of the closed discrete subset $\A_n^\circ\setminus\A_{n+1}^\circ$ of the decomposition space $\A_n$. The neighborhood $U'_n$ will be constructed as follows.

Observe that each element $a\in \A^\circ_n\setminus\A^\circ_{n+1}\subset\A_n$ is a compact subset of the space $X$, equal to its own preimage $q^{-1}_{\A_n}(a)$ under the quotient map $q_{\A_n}:X\to\A_n$. The condition $(2_{n})$ of the inductive construction guarantees that $\diam(a)<2^{-n+1}$. The same is true for any point  $b\in\BB^\circ_n\setminus\BB_{n+1}^\circ=h_n(\A^\circ_n\setminus\A^\circ_{n+1})\subset\BB_n$: it coincides with its own preimage $q_{\BB_n}^{-1}(b)\subset X$ and has diameter $\diam(b)<2^{-n+1}$. Since the non-degeneracy set $\bar B_n=\bigcup(\BB_0^n)^\circ$ of the map $q^{\BB_n}_{\BB_0}:\BB_n\to\BB$ is disjoint  with the closed discrete subset $\BB^\circ_n\setminus \BB_{n+1}^\circ\subset \BB_{n}$, for every point $b\in \BB^\circ_n\setminus\BB_{n+1}^\circ$ we can choose a neighborhood $U_n(b)\subset \BB_n$ such that
\begin{itemize}
\item $U_n(b)\cap \bar B_n=\emptyset$,
\item $\diam\big(q^{-1}_{\BB_n}(U_n(b))\big)<2^{-n+1}$;
\item $\diam\big(q^{-1}_{\A_n}(h_n^{-1}(U_n(b)))\big)<2^{-n+1}$, and
\item $U_n(b)=(q^{\BB_n}_{\BB_0})^{-1}(W_n(b))$ for some open set $W_n(b)\subset\bigcup\W\subset\BB$\newline that has $\rho$-diameter $\diam\big(W_n(b)\big)<2^{-n-1}\cdot\inf \e\circ\varphi_0\circ\varphi_n^{-1}(W_n(b))$.
\end{itemize}
Since the set $\BB^\circ_n\setminus\BB_{n+1}^\circ$ is closed and discrete in the (collectionwise normal) decomposition space $\BB_n$, we can additionally assume that the indexed family $\{U_n(b):b\in \BB^\circ_n\setminus\BB_{n+1}^\circ\}$ is discrete in $\BB_n$.

Then
\begin{enumerate}
\item $U_n=\bigcup\{U_n(b):b\in \BB^\circ_n\setminus\BB_{n+1}^\circ\}$ is an open neighborhood of the closed discrete subset $\BB^\circ_n\setminus\BB_{n+1}^\circ$ in the decomposition space $\BB_n$,
\item $W_n=\bigcup\{W_n(b):b\in \BB^\circ_n\setminus\BB_{n+1}^\circ\}$ is an open neighborhood of the closed discrete subset $\BB^\circ_n\setminus\BB_{n+1}^\circ$ in the decomposition space $\BB$,
\item $U_n'=h_n^{-1}(U_n)$ is an open neighborhood of the closed discrete subset $\A^\circ_n\setminus\A_{n+1}^\circ=h_n^{-1}(\BB^\circ_n\setminus\BB_{n+1}^\circ)$ in the decomposition space $\A_n$, and
\item $W_n'=\varphi_n^{-1}(W_n)$ is an open neighborhood of the closed discrete subset $\A^\circ_n\setminus\A_{n+1}^\circ=\varphi_n^{-1}(\BB^\circ_n\setminus\BB_{n+1}^\circ)$ in the decomposition space $\A$.
\end{enumerate}
The choice of the neighborhoods $U_n(b)$, $b\in\BB_n^\circ\setminus\BB_{n+1}^\circ$, guarantees that $U_n\cap\bar B_n=\emptyset$, which implies $U_n'\cap\bar A_n=\emptyset$.

These sets fit into the following commutative diagram:
$$\xymatrix{
&&\A_{n+1}\ar[d]_{q_{\A_n}^{\A_{n+1}}}&\BB_{n+1}\ar[d]^{q_{\BB_n}^{\BB_{n+1}}}
&&\\
\A_n^\circ\setminus\A_{n+1}^\circ\ar[r]\ar[d]&U_n'\ar[r]\ar[d]&\A_n\ar[r]^{h_n}\ar[d]_{q^{\A_n}_{\A_0}}&
\BB_n\ar[d]^{q^{\BB_n}_{\BB_0}}&U_n\ar[l]\ar[d]
&\BB_n^\circ\setminus\BB_{n+1}^\circ\ar[l]\ar[d]\\
\A_n^\circ\setminus\A_{n+1}^\circ\ar[r]&W_n'\ar[r]&\A\ar[r]^{\varphi_n}&\BB&W_n\ar[l]
&\BB_n^\circ\setminus\BB_{n+1}^\circ\ar[l]\\
}
$$

It follows that $U_n$ is an open neighborhood of the non-degeneracy set $\BB_n^\circ\setminus\BB_{n+1}^\circ$ of the map $q^{\BB_{n+1}}_{\BB_n}:\BB_{n+1}\to\BB_n$ while $U_n'$ is an open neighborhood of the non-degeneracy set $\A_n^\circ\setminus \A_{n+1}^\circ$ of the map $q^{\A_{n+1}}_{\A_n}:\A_{n+1}\to\A_n$.

The shrinkability of the decomposition $\A$ (which follows from the $\K$-tameness of $\A$) implies the shrinkability of the decomposition $\A_{n+1}\prec\A$. Then Theorem~\ref{t:shrink} implies that the quotient map $q_{\A_{n+1}}:X\to\A_{n+1}$ is a near homeomorphism and hence the decomposition space $\A_{n+1}=X/\A_{n+1}$ is homeomorphic to $X$. So, we can consider the tame family $\K(\A_{n+1})=\{f(K):K\in\K,\; f\in\HH(X,\A_{n+1})\}$ of compact subsets of $\A_{n+1}$.

We claim that $(\A_{n}^{n+1})^\circ\subset\K(\A_{n+1})$. Fix any set $A_{n+1}\in(\A_n^{n+1})^\circ$ and consider its preimage $A=q_{\A_{n+1}}^{-1}(A_{n+1})\in \A^\circ_n\setminus\A^\circ_{n+1}$ in $X$. Observe that $A_{n+1}$ is a compact subset of the decomposition space $\A_{n+1}$, disjoint with its non-degeneracy part $\A_{n+1}^\circ$. Since $A\in\A$, the open set $S=X\setminus A$ is $\A$-saturated.
The strong shrinkability of the decomposition $\A$ implies the shrinkability of the decompositions $\A|S$ and $\A_{n+1}|S$. Then Theorem~\ref{t:shrink} and Lemma~\ref{l:approx} imply that the quotient map $q_{\A_{n+1}}:X\to\A_{n+1}$ can be approximated by a homeomorphism $h:X\to \A_{n+1}$ such that $h(A)=A_{n+1}$, which means that the pairs $(X,A)$ and $(\A_{n+1},A_{n+1})$ are homeomorphic and hence $A_{n+1}\in\K(\A_{n+1})$.
So, $\A_n^{n+1}$ is a discrete $\K(\A_{n+1})$-tame decomposition of the space $\A_{n+1}$.

By analogy, we can show that the decomposition $\BB_n^{n+1}$ of the space $\BB_{n+1}$ is discrete and $\K(\BB_{n+1})$-tame for the tame family  $\K(\BB_{n+1})=\{f(K):K\in\K,\;\;f\in\HH(X,\BB_{n+1})\}$ of compact subsets of the decomposition space $\BB_{n+1}$ (which is homeomorphic to $X$).

Now one can apply Theorem~\ref{t4}, and approximate the homeomorphism $h_n:\A_n\to \BB_n$ by a $(\A_n^{n+1},\BB_n^{n+1})$-liftable homeomorphism $\tilde h_n:\A_n\to\BB_n$ such that $(\tilde h_n,h_n)\prec\U_n$ where $\U_n=\{U_n(b):b\in\BB_n^\circ\setminus\BB_{n+1}^\circ\}$. The relation $(\tilde h_n,h_n)\prec\U_n$ implies that $\tilde h_n|X\setminus U_n'=h_n|X\setminus U'_n$ and hence $\tilde h_n|\bar A_n=h_n|\bar A_n$.
The homeomorphism $\tilde h_n$ can be lifted to a homeomorphism $\tilde h_{n+1}:\A_{n+1}\to\BB_{n+1}$ making the following diagram commutative:
$$\xymatrix{
\A_{n+1}\ar[d]_{q_{\A_n}^{\A_{n+1}}}\ar[r]^{\tilde h_{n+1}}&\BB_{n+1}\ar[d]^{q_{\BB_n}^{\BB_{n+1}}}\\
\A_n\ar[r]_{\tilde h_n}&\BB_n
}
$$
Since the homeomorphism $\tilde h_n$ is $(\A_n^{n+1},\BB_n^{n+1})$-liftable it maps the non-degeneracy set $\A_n^\circ\setminus\A_{n+1}^\circ$ of the map $q^{\A_{n+1}}_{\A_n}$ onto the non-degeneracy set
 $\BB_n^\circ\setminus\BB_{n+1}^\circ$ of the map $q^{\BB_{n+1}}_{\BB_n}$. This fact, combined with the equality $\tilde h_n|\bar A_n=h_n|A_n$, implies $\tilde h_{n+1}(\bar A_{n+1})=\bar B_{n+1}$.

Now we shall approximate the homeomorphism $\tilde h_{n+1}$ by a homeomorphism $h_{n+1}:\A_{n+1}\to\BB_{n+1}$ such that $h_{n+1}|\bar A_{n+1}=h_{n+1}|\bar A_{n+1}$ and $h_{n+1}(\A_{n+1}^\circ)=\BB_{n+1}^\circ$.

For this, for every point $b\in \BB^\circ_n\setminus \BB^\circ_{n+1}\subset\BB_n$, consider the open set $U_n(b)\setminus\{b\}$ and its preimage $V_{n+1}(b)=(q_{\BB_n}^{\BB_{n+1}})^{-1}(U_n(b)\setminus\{b\})$ in $\BB_{n+1}$. Then $\V_{n+1}=\{V_{n+1}(b):b\in\BB^\circ_n\setminus\BB^\circ_{n+1}\}$ is an open cover of the open subset $V_{n+1}=\bigcup\V_{n+1}\subset\BB_{n+1}$, which coincides with the set $(q_{\BB_n}^{\BB_{n+1}})^{-1}\big(U_n\setminus(\BB^\circ_n\setminus\BB^\circ_{n+1})\big)$and does not intersect the closed subset $\bar B_{n+1}=\bigcup(\BB^{n+1}_0)^\circ$ of the decomposition space $\BB_{n+1}$.
It follows that the open subset $V_{n+1}'=\tilde h_{n+1}^{-1}(V_{n+1})$ of the decomposition space  $\A_{n+1}$ coincides with the set
$(q_{\A_n}^{\A_{n+1}})^{-1}\big(U'_n\setminus(\A^\circ_n\setminus\A^\circ_{n+1})\big)$ and does not intersect the closed subset $\bar A_{n+1}=\bigcup(\A_0^{n+1})^\circ$ of  $\A_{n+1}$.

The density of the decomposition $\A_0=\A$ implies that the set $q_{\A_{n+1}}\big(\bigcup\A_0^\circ \big)$ is dense in $\A_{n+1}$ and the set $\A_{n+1}^\circ=q_{\A_{n+1}}\big(\bigcup\A_0^\circ\big)\setminus\bar A_n$ is dense in $\A_{n+1}\setminus\bar A_n$. Taking into account that the decomposition $\A_{n+1}$ is vanishing, we conclude that its non-degeneracy part  $\A^\circ_{n+1}=\bigcup_{k\in\w}\A^\circ_{n+1,2^{-k}}$ is $\sigma$-discrete in $\A_{n+1}$. Then $\A_{n+1}^\circ\cap V'_{n+1}$ is a dense $\sigma$-discrete subset in $V'_{n+1}$. By analogy we can show that $\BB_{n+1}^\circ\cap V_{n+1}$ is a dense $\sigma$-discrete subset in $V_{n+1}$. Applying Theorem~\ref{t:dense}, we can approximate the homeomorphism $\tilde h_{n+1}$ by a homeomorphism $h_{n+1}:\A_{n+1}\to\BB_{n+1}$ such that $h_{n+1}(V'_{n+1}\cap \A^\circ_{n+1})=V_{n+1}\cap\BB^\circ_{n+1}$ and $(h_{n+1},\tilde h_{n+1})\prec\V_{n+1}$, which implies that the homeomorphisms $h_{n+1}$ and $\tilde h_{n+1}$ coincide on the set $X\setminus V'_{n+1}\supset\bar A_{n+1}$.

We claim that the homeomorphism $h_{n+1}$ is $(\A^{n+1}_0,\BB^{n+1}_0)$-factorizable. This will follow
as soon as we check that
for every sets $A\in \A^{n+1}_0$ and $B\in \BB^{n+1}_0$ the sets $q^{\BB_{n+1}}_{\BB_0}\circ h_{n+1}(A)\subset\BB$ and $q^{\A_{n+1}}_{\A_0}\circ h^{-1}_{n+1}(B)\subset\A$ are singletons.
First we check that the set  $q^{\BB_{n+1}}_{\BB_0}\circ h_{n+1}(A)\subset\BB$ is a singleton.
This is clear if $A$ is a singleton. So, we assume that $A$ is not a singleton, in which case $A\subset\bigcup(\A_0^{n+1})^\circ=\bar A_{n+1}$, $h_{n+1}|A=\tilde h_{n+1}|A$, and
$$
q_{\BB_0}^{\BB_{n+1}}\circ h_{n+1}(A)
=q_{\BB_0}^{\BB_{n}}\circ q_{\BB_n}^{\BB_{n+1}}\circ \tilde h_{n+1}(A)
=q_{\BB_0}^{\BB_{n}}\circ \tilde h_n\circ q_{\A_n}^{\A_{n+1}}(A)=q_{\BB_0}^{\BB_{n}}\circ h_n\circ q_{\A_n}^{\A_{n+1}}(A).$$
Observe that the set $q_{\A_n}^{\A_{n+1}}(A)$ is an element of the decomposition $\A_0^n$ of the decomposition space $\A_n$. The condition $(3_n)$ of the inductive assumption guarantees that the homeomorphism $h_n$ is $(\A_0^n,\BB_0^n)$-factorizable, which implies that the set $$q^{\BB_n}_{\BB_0}\circ h_n\circ q_{\A_n}^{\A_{n+1}}(A)=q_{\BB}^{\BB_{n+1}}\circ h_{n+1}(A)$$ is a singleton.
By analogy we can check that for every set  $B\in \BB^{n+1}_0$ the set $q^{\BB_{n+1}}_{\A}\circ h^{-1}_{n+1}(B)$ is a singleton in $\A$. This implies that the homeomorphism $h_{n+1}$ is $(\A_0^{n+1},\BB_0^{n+1})$-factorizable and hence there is a homeomorphism $\varphi_{n+1}:\A\to\BB$ such that $q_{\BB_0}^{\BB_{n+1}}\circ h_{n+1}=\varphi_{n+1}\circ q_{\A_0}^{\A_{n+1}}$. So, the condition $(3_{n+1})$ of the inductive construction is satisfied.

To prove the condition $(4_{n+1})$, we need to prove that $\rho(\varphi_{n+1}(a),\varphi_n(a))\le2^{-n-1}\cdot\e\circ\varphi_0(a)$ for each $a\in\A$. This inequality follows from the equality $\varphi_{n+1}(a)=\varphi_n(a)$ if $a\in\A\setminus W_n'$. If $a\in W'_n$, then $\varphi_{n+1}(a),\varphi_n(a)\in W_n(b)$ for some $b\in \BB_n^\circ\setminus\BB_{n+1}^\circ$ and hence $$\rho(\varphi_{n+1}(a),\varphi_n(a))\le\diam\big( W_n(b)\big)\le 2^{-n-1}\cdot\inf \e\circ\varphi_0\circ\varphi_n^{-1}(W_n(b))\le 2^{-n-1}\cdot\e\circ\varphi_0(a).$$

It follows from the construction of the homeomorphisms $h_{n+1}$ and the choice of the neighborhoods $W_n(b)$, $b\in\BB^\circ_n\setminus\BB_{n+1}^\circ$, that
the homeomorphisms $\varphi_{n+1}$ and $\varphi_n$ coincide on the set $(\A\setminus W'_n)\cup (\A_n^\circ\setminus\A_{n+1}^\circ)\supset\A_0^\circ\setminus\A_{n+1}^\circ$. So, the inductive condition $(5_{n+1})$ holds.

Taking into account that the homeomorphism $\tilde h_n$ coincides with the homeomorphism $h_n$ on the set $\A^\circ_n\setminus\A^\circ_{n+1}$ and the homeomorphism $h_{n+1}$ coincides with the lift $\tilde h_{n+1}$ of $\tilde h_n$ on the set $(q_{\A_n}^{\A_{n+1}})^{-1}(\A^\circ_n\setminus\A^\circ_{n+1})=(\A_n^{n+1})^\circ$, we conclude that $\varphi_{n+1}(\A^\circ_n\setminus\A^\circ_{n+1})=\varphi_n(\A^\circ_n\setminus\A^\circ_{n+1})=\BB_n^\circ\setminus\BB_{n+1}^\circ$.
To finish the proof of the condition $(6_{n+1})$, observe that the equalities $\varphi_{n+1}|\A\setminus W_n'=\varphi_n|\A\setminus W'_n$ and $\varphi_{n+1}(W'_n)=W_n$ and the inductive assumption $(6_n)$ imply
$$\varphi_{n+1}(\A^\circ_{n}\setminus W'_n)=\varphi_n(\A^\circ_n\setminus W'_n)=\BB^\circ_n\setminus W_n.$$ On the other hand, the equality $h_{n+1}(\A^\circ_{n+1}\cap V_{n+1}')=\BB^\circ_{n+1}\cap V_{n+1}$ and the definition of the open sets $V_{n+1}$ and $V_{n+1}'$ imply that $\varphi_{n+1}(\A^\circ_{n+1}\cap W'_n)=\BB^\circ_{n+1}\cap W_n$. So, $\varphi_{n+1}(\A_{n+1}^\circ)=\BB_{n+1}^\circ$, which means that the condition $(6_{n+1})$ holds.

To complete the inductive step, it remains to check that the multivalued map $\Phi_{n+1}=q_{\BB_{n+1}}^{-1}\circ h_{n+1}\circ q_{\A_{n+1}}:X\setmap X$ satisfies the conditions $(7_{n+1})$ and $(8_{n+1})$. To see that the condition $(7_{n+1})$ holds, observe that
the map $q_{\A_n}^{\A_{n+1}}:\A_{n+1}\to\A_n$ is injective on the set $\bar A'_n=q_{\A_{n+1}}\big(\bigcup(\A_0\setminus\A_n)\big)$ and the map  $q_{\BB_n}^{\BB_{n+1}}:\BB_{n+1}\to\BB_n$ is injective on the set $\bar B'_n=q_{\BB_{n+1}}\big(\bigcup(\BB_0\setminus\BB_n)\big)$. Taking into account that $h_{n+1}|\bar A_n'=\tilde h_{n+1}|\bar A'_n$ and  $\tilde h_n|\bar A_n=h_n|\bar A_n$, we conclude that $$h_{n+1}|\bar A'_n=\tilde h_{n+1}|\bar A_n'=(q_{\BB_n}^{\BB_{n+1}})^{-1}\circ \tilde h_n\circ q_{\A_n}^{\A_{n+1}}|\bar A'_n=(q_{\BB_n}^{\BB_{n+1}})^{-1}\circ  h_n\circ q_{\A_n}^{\A_{n+1}}|\bar A'_n$$and hence for every $x\in\bigcup(\A_0^\circ\setminus\A_n^\circ)$ we get
$$
\begin{aligned}
\Phi_{n+1}(x)&=q_{\BB_{n+1}}^{-1}\circ h_{n+1}\circ q_{\A_{n+1}}(x)=q_{\BB_{n+1}}^{-1}\circ \tilde h_{n+1}\circ q_{\A_{n+1}}(x)=\\
&=q_{\BB_{n+1}}^{-1}\circ (q_{\BB_n}^{\BB_{n+1}})^{-1}\circ \tilde  h_{n}\circ q_{\A_n}^{\A_{n+1}}\circ q_{\A_{n+1}}(x)=\\
&=\big(q^{\BB_{n+1}}_{\BB_n}\circ q_{\BB_{n+1}}\big)^{-1}\circ h_{n}\circ q_{\A_n}(x) =q_{\BB_n}^{-1}\circ h_{n}\circ q_{\A_n}(x)=\Phi_n(x).
\end{aligned}
$$
So, the condition $(7_{n+1})$ holds.

To check the condition $(8_{n+1})$, fix any point $x\in X$. If the projection $a=q_\A(x)\in\A_0$ does not belong to the open set $W'_n$, then $\Phi_{n}(x)=\Phi_{n+1}(x)\in \BB_{n+1}$ and hence $\diam(\Phi_n(x)\cup\Phi_{n+1}(x))=\diam(\Phi_{n+1}(x))<2^{-n}$ by the condition $(2_{n+1})$ of the inductive construction. So, we assume that $a\in W'_n$ and hence $\varphi_n(a),\varphi_{n+1}(a)\in W_n(b)$ for some element $b\in\BB^\circ_n\setminus\BB^\circ_{n+1}$. The choice of the neighborhood $W_n(b)$ guarantees that the set $q_\BB^{-1}(W_n(b))$ has diameter $<2^{-n+1}$. Taking into account that $$\Phi_n(x)\cup\Phi_{n+1}(x)\subset q_{\BB}^{-1}(\{\varphi_n(a),\varphi_{n+1}(a)\})\subset q_{\BB}^{-1}(W_n(b)),$$ we obtain the desirable inequality
$$\diam\big(\Phi_n(x)\cup\Phi_{n+1}(x)\big)\le \diam \big(q_{\BB}^{-1}(W_n(b))\big)<2^{-n+1}.$$
By analogy we can prove that $\diam\big(\Phi_n^{-1}\cup\Phi_{n+1}^{-1}(x)\big)<2^{-n+1}$.
This completes the inductive step.
\smallskip

After completing the inductive construction, we obtain the sequences of decompositions $(\A_n)_{n\in\w}$, $(\BB_n)_{n\in\w}$ of $X$, the sequences of homeomorphisms $(h_n:\A_n\to\BB_n)_{n\in\w}$, $(\varphi_n:\A\to\BB)_{n\in\w}$ and the sequence $(\Phi_n:X\setmap X)_{n\in\w}$ of multivalued functions, satisfying the conditions $(1_n)$--$(8_n)$, $n\in\IN$, of the inductive construction.

Taking the limit $\Phi=\lim_{n\to\infty}\Phi_n$ of the multivalued functions $\Phi_n$ we shall obtain a $(\A,\BB)$-factorizable homeomorphism $\Phi:X\to X$ inducing a $(\A,\BB)$-liftable homeomorphism $\varphi:\A\to\BB$ of the decomposition spaces.

To define the map $\Phi$, consider for every $x\in X$ the sequence $(\Phi_n(x))_{n\in\w}$ of compact subsets of the space $X$. The conditions $(8_n)$, $n\in\IN$, of the inductive construction guarantee that this sequence is Cauchy in the hyperspace $\exp(X)$ of $X$ endowed with the Hausdorff metric $d_H$, which is complete according to \cite[4.5.23]{En}. Let us recall that the {\em hyperspace} $\exp(X)$ is the space of non-empty compact subsets of $X$, endowed with the {\em Hausdorff metric} $d_H$ defined by the (well-known) formula
$$d_H(A,B)=\max\{\max_{a\in A}d(a,B),\max_{b\in B}d(A,b)\}\mbox{ where $A,B\in\exp(X)$}.$$
We shall identify the metric space $(X,d)$ with the subspace of singletons in $(\exp(X),d_H)$.

The completeness of the hyperspace $(\exp(X),d_H)$ guarantees that the Cauchy sequence $(\Phi_n(x))_{n\in\w}$ has the limit $\Phi(x)=\lim\limits_{n\to\infty}\Phi_n(x)$ in $\exp(X)$. Moreover, the conditions $(8_n)$, $n\in\IN$, imply that
\begin{equation}\label{lim}
d_H(\Phi(x),\Phi_n(x))\le \sum_{k=n}^\infty d_H(\Phi_{k+1}(x),\Phi_k(x))\le \sum_{k=n}^\infty\diam(\Phi_{k+1}(x)\cup\Phi_k(x))<\sum_{k=n}^\infty 2^{-k+1}=2^{-n+2}
\end{equation} for every $n\in\IN$.

Also the conditions $(8_n)$, $n\in\IN$, of the inductive construction imply that $\Phi(x)=\lim_{n\to\infty}\Phi_n(x)$ is a singleton. So, $\Phi:x\mapsto \Phi(x)$ can be thought as a usual singlevalued function $\Phi:X\to X\subset \exp(X)$.

\begin{claim}\label{cl8.2} The function $\Phi:X\to X$ is continuous.
\end{claim}

\begin{proof} Given any point $x_0\in X$ and $\epsilon>0$, we need to find a neighborhood $O(x_0)\subset X$  such that $\Phi(O(x_0))\subset O_d(\Phi(x_0),\epsilon)$ where $O_d(y,\epsilon)=\{x\in X:d(x,y)<\epsilon\}$ denotes the $\epsilon$-ball centered at a point $y\in X$. Find $n\in\IN$ such that $2^{-n+5}<\epsilon$ and consider the multivalued function $\Phi_n=q_{\BB_n}^{-1}\circ h_n\circ q_{\A_n}:X\setmap X$. Consider the point $a=q_{\A_n}(x_0)\in\A_n$ and its image $b=h_n(A)\in\BB_n$, which is a compact subset of $X$. Since the quotient map $q_{\BB_n}:X\to  \BB_n$ is closed, the point $b\in \BB_n$ has an open neighborhood $O(b)\subset \BB_n$ such that $q_{\BB_n}^{-1}(O(b))\subset O_d(b,2^{-n})$. By the continuity of the homeomorphism $h_n:\A_n\to\BB_n$, the point $a\in\A_n$ has a neighborhood $O(a)\subset\A_n$ such that $h_n(O(a))\subset O(b)$. The continuity of the quotient projection $q_{\A_n}$ implies that $O(x_0)=q_{\A_n}^{-1}(O(a))$ is an open neighborhood of the point $x_0\in q_{\A_n}^{-1}(a)$.

We claim that $d(\Phi(x),\Phi(x_0))<\epsilon$ for every $x\in O(x_0)$. Observe that $\Phi_n(x_0)\cup \Phi_n(x)\subset q_{\BB_n}^{-1}\circ h_n\circ q_{\A_n}(\{x,x_0\})\subset q_{\BB_n}^{-1}(h_n(O(a))\subset q_{\BB_n}^{-1}(O(b))\subset O_d(b,2^{-n})$. Now the upper bound (\ref{lim}) implies that
$$\Phi(x)\cup\Phi(x_0)\subset O_d(\Phi_n(x)\cup\Phi_n(x_0),2^{-n+2})\subset O_d(b,2^{-n}+2^{-n+2})\subset O_d(b,2^{-n+3}).$$
Since $b\in \BB_n$, the condition $(2_{n})$ of the inductive construction guarantees that $\diam(b)< 2^{-n+1}$. Consequently, $$d(\Phi(x),\Phi(x_0))\le\diam\big( O_d(b,2^{-n+3})\big)\le \diam(b)+2\cdot 2^{-n+3}\le2^{-n+1}+2^{-n+4}<2^{-n+5}<\epsilon.$$
\end{proof}

\begin{claim}\label{cl8.3} There exists a continuous function $\varphi:\A\to\BB$ such that $q_\BB\circ\Phi=\varphi\circ q_\A$ and $\varphi|\A^\circ_0\setminus\A^\circ_{n}=\varphi_{n}|\A^\circ_0\setminus \A^\circ_{n}$ for all $n\in\IN$.
\end{claim}

\begin{proof} To define the function $\varphi:\A\to\BB$, we shall show that for each element $a\in\A$ the set $q_\BB\circ\Phi(a)\subset\BB$ is a singleton. This is trivially true if the compact subset $a$ of $X$ is a singleton. So, we assume that $a$ is not a singleton and hence $a\in\A^\circ_{n-1}\setminus\A^\circ_{n}$ for some $n\in\IN$. In this case
$a\subset\bigcup(\A_0^\circ\setminus\A_{n}^\circ)$ and hence $\Phi|a=\Phi_{n}|a$ by the conditions $(7_k)$, $k>n$, of the inductive construction. Now we see that the set
$$
\begin{aligned}
q_\BB\circ \Phi(a)&=q_\BB\circ \Phi_{n}(a)=q_\BB\circ q_{\BB_{n}}^{-1}\circ h_n\circ q_{\A_n}(a)=\\
&=q_{\BB}^{\BB_n}\circ q_{\BB_n}\circ q_{\BB_n}^{-1}\circ h_n\circ q_{\A_n}(a)=\\
&=q_{\BB}^{\BB_n}\circ h_n\circ q_{\A_n}(a)=\varphi_n\circ q_{\A}^{\A_n}\circ q_{\A_n}(a)=\\
&=\varphi_n\circ q_\A(a)=\varphi_n(\{a\})=\{\varphi_n(a)\}
\end{aligned}
$$ is a singleton.
So, there is a unique function $\varphi:\A\to\BB$ making the following square commutative:
$$\xymatrix{
X\ar[d]_{q_\A}\ar[r]^{\Phi}&X\ar[d]^{q_\BB}\\
\A\ar[r]_{\varphi}&\BB.
}$$
Taking into account that the functions $\Phi$, $q_\BB$ are continuous, and the function $q_\A$ is closed, we conclude that the function $\varphi$ is continuous.
\end{proof}

By analogy with the proofs of Claims~\ref{cl8.2} and \ref{cl8.3} we can prove

\begin{claim} \begin{enumerate}
\item For every point $x\in X$ the sequence $(\Phi_n^{-1}(x))_{n\in\w}$ of compact subsets of $X$ converges in the hyperspace $(\exp(X),d_H)$ to some singleton $\Psi(x)\subset X$,
\item the function $\Psi:X\to X\subset\exp(X)$, $\Psi:x\mapsto\Psi(x)$, is continuous, and
\item the function $\Psi$ is $(\BB,\A)$-factorizable, which means that the square
$$\xymatrix{
X\ar[r]^{\Psi}\ar[d]_{q_{\BB}}&X\ar[d]^{q_{\A}}\\
\BB\ar[r]_{\psi}&\A
}
$$is commutative for some continuous function $\psi:\BB\to\A$.
\end{enumerate}
\end{claim}

Next, we show that the functions $\Phi$ and $\Psi$ are inverse of each other.

\begin{claim}  $\Phi\circ\Psi(x)=\lim_{n\to\infty}\Phi_n\circ \Phi_n^{-1}(x)$ for every $x\in X$.
\end{claim}

\begin{proof} Given any $\epsilon>0$ we need to find $m\in\IN$ such that $d_H(\Phi\circ\Psi(x),\Phi_n\circ\Phi_n^{-1}(x))<\epsilon$ for all $n\ge m$.

By the continuity of the map $\Phi$ at the singleton $\Psi(x)$, there is $\delta>0$ such that $\Phi\big(O_d(\Psi(x),\delta)\big)\subset O_d(\Phi\circ\Psi(x),\epsilon/2)$. Choose $m\in\IN$ so large that $2^{-m+3}<\min\{\epsilon,\delta\}$ and take any $n\ge m$. By analogy with the equality (\ref{lim}), we can prove that $d_H(\Psi(x),\Phi_n^{-1}(x))<2^{-n+2}$ and hence
$\Phi_n^{-1}(x)\subset O_d(\Psi(x),2^{-n+2})\subset O_d(\Psi(x),\delta)$. The choice of $\delta$ guarantees that $\Phi\circ \Phi^{-1}_n(x)\subset \Phi\big(O_d(\Psi(x),\delta)\big)\subset O_d(\Phi\circ\Psi(x),\epsilon/2)$, which implies $d_H(\Phi\circ \Phi^{-1}_n(x),\Phi\circ \Psi(x))<\epsilon/2$. On the other hand, the equality (\ref{lim}) implies that
$$d_H(\Phi_n\circ\Phi^{-1}_n(x),\Phi\circ\Phi^{-1}_n(x))<2^{-n+2}<\epsilon/2$$and hence
$$d_H(\Phi_n\circ\Phi^{-1}_n(x),\Phi\circ\Psi(x))\le d_H(\Phi_n\circ\Phi^{-1}_n(x),\Phi\circ\Phi^{-1}_n(x))+d_H(\Phi\circ \Phi_n^{-1}(x),\Phi\circ\Psi(x))<\epsilon/2+\epsilon/2=\epsilon.$$
\end{proof}

\begin{claim} $\Phi\circ \Psi(x)=\{x\}$ for all $x\in X$.
\end{claim}

\begin{proof} For every $n\in\IN$, the definition of the multivalued function $\Phi_n$ implies that $$x\in \Phi_n\circ\Phi_n^{-1}(x)\subset q_{\BB_n}^{-1}\circ q_{\BB_n}(x)=q_{\BB_n}(x)\in\BB_n.$$
The condition $(2_{n-1})$ of the inductive construction guarantees that
$$\diam(\Phi_n\circ\Phi_n^{-1}(x))\le \diam\big( q_{\BB_n}(x)\big)<2^{-n+1},$$ which implies that $\Phi_n\circ\Phi_n^{-1}(x)\subset O_d(x,2^{-n+1})$ and hence $\Phi\circ\Psi(x)=\lim_{n\to\infty}\Phi_n\circ\Phi_n^{-1}(x)=\{x\}$.
\end{proof}

By analogy we can prove that $\Psi\circ\Phi(x)=\{x\}$ for all $x\in X$.
So, $\Phi\circ\Psi=\id_X=\Psi\circ\Phi$.
\smallskip

Now consider the commutative diagram
$$
\xymatrix{
X\ar[r]^{\Phi}\ar[d]_{q_{\A}}&X\ar[r]^\Psi\ar[d]^{q_{\BB}}&X\ar[d]^{q_\A}\\
\A\ar[r]_{\varphi}&\BB\ar[r]_{\psi}&\A
}$$
and observe that $\psi\circ\varphi:\A\to\A$ is a unique map such that $q_\A\circ\id_X=q_\A\circ(\Psi\circ\Phi)=(\psi\circ\phi)\circ q_\A$, which implies that $\psi\circ\phi=\id_\A$. By analogy we can prove that $\phi\circ\psi=\id_\BB$. This means that $\varphi:\A\to\BB$ is a $(\A,\BB)$-liftable homeomorphism with the inverse $\varphi^{-1}=\psi$.

To finish the proof of Theorem~\ref{main2}, it remains to check that the homeomorphism $\varphi$ is $\W$-near to the homeomorphism $\varphi_0$. By the choice of the function $\e:\BB\to[0,1]$ this will follow as soon as we check that $\rho(\varphi,\varphi_0)\le\e\circ\varphi_0$.

By the density of the set $\A^\circ$ in $\A$ and the continuity of the functions $\varphi$, $\varphi_0$, and $\e$, it suffices to check that $\rho(\varphi|\A^\circ,\varphi_0|\A^\circ)\le \e\circ\varphi_0|\A^\circ$. Given any point $a\in\A^\circ$,
find a (unique) number $n\in\IN$ with $a\in\A^\circ_{n-1}\setminus\A^\circ_n$. Then $\varphi(a)=\varphi_n(a)$ and hence
$$\rho(\varphi(a),\varphi_0(a))=\rho(\varphi_n(a),\varphi_0(a))\le
\sum_{k=1}^n\rho(\varphi_k(a),\varphi_{k-1}(a))\le\sum_{k=1}^n2^{-k}\e\circ\varphi_0(a)\le\e\circ\varphi_0(a)$$
by the conditions $(4_k)$, $k\in\IN$, of the inductive construction.
\end{proof}

\section{Proof of Theorem~\ref{main2}}\label{s:main2}

In this subsection we shall deduce Theorem~\ref{main2} from Theorems~\ref{t5} and \ref{t:dense}. Given a tame collection $\K$ of compact subsets of a strongly locally homogeneous completely metrizable space $X$ and two dense $\K$-tame decompositions $\A,\BB$ of the space $X$, we need to show that the set of $(\A,\BB)$-liftable homeomorphisms in dense in the homeomorphism space $\HH(\A,\BB)$.

This will be done as soon as for each homeomorphism $f:\A\to\BB$ and an open cover $\U$ of the decomposition space $\BB=X/\BB$ we find an $(\A,\BB)$-liftable homeomorphism $h:\A\to\BB$ which is $\U$-near to $f$. By Lemma~\ref{l1}, the decomposition space $\BB$ is  metrizable and hence paracompact. So, we can find an open cover $\V$ of $\BB$ such that  $\St(\V)\prec\U$.

First we shall find a homeomorphism $g:\A\to\BB$ such that $(g,f)\prec\V$ and $g(\A^\circ)=\BB^\circ$.
Fix any complete metric $d$ generating the topology of the completely metrizable space $X$. Since the decomposition $\BB$ is vanishing, for every $\epsilon>0$ the subfamily $\BB^\circ_\e=\{B\in\BB:\diam(B)>\epsilon\}$ is discrete in $X$ and hence $\BB^\circ_\epsilon$ is a closed discrete subset in the decomposition space $\BB$. Since $\BB^\circ=\bigcup_{n\in\w}\BB^\circ_{2^{-n}}$, we see that the non-degeneracy part $\BB^\circ$ of the (dense) decomposition $\BB$ is $\sigma$-discrete (and dense) in $\BB$.

By analogy we can show that the non-degeneracy part $\A^\circ$ of the decomposition $\A$ is dense and $\sigma$-discrete in the decomposition space $\A$. Then $f(\A^\circ)$ is a dense $\sigma$-discrete subset of the decomposition space $\BB$.

By Theorem~\ref{t:shrink}, the quotient map $q_\BB:X\to X/\BB$ is a strong near homeomorphism, which implies that the decomposition space $\BB$ is homeomorphic to $X$ and hence is strongly locally homogeneous and completely metrizable. Now it is legal to apply Theorem~\ref{t:dense} and find a homeomorphism $h:\BB\to\BB$ such that $(h,\id)\prec\V$ and $h\big(f(\A^\circ)\big)=\BB^\circ$. Then the homeomorphism $g=h\circ f:\A\to\BB$ maps $\A^\circ$ onto $\BB^\circ$ and is $\V$-near to $f$.

Since $g(\A^\circ)=\BB^\circ$, the homeomorphism $g:\A\to\BB$  belongs to the space $\HH^\circ(\A^\circ,\BB^\circ)$. Applying Theorem~\ref{t5}, find a $(\A,\BB)$-liftable homeomorphism $\varphi:\A\to\BB$ such that $(\varphi,g)\prec\V$. It follows from $(\varphi,g)\prec\V$ and $(g,f)\prec\V$ that $(\varphi,f)\prec\St(\V)\prec\U$. So, $\varphi:\A\to\BB$ is a required $(\A,\BB)$-liftable homeomorphism, which is $\U$-near to the homeomorphism $f$.

\section{Existence of $\K$-tame decompositions}\label{s:exist}

In this section we shall prove Theorem~\ref{t:exist}. Let $(X,d)$ be a metric space and $\K$ be a tame family of compact subsets of $X$ containing more than one point. Given a non-empty open subset $U\subset X$, we need to construct a $\K$-tame decomposition $\DD$ of $X$ such that $\bigcup\DD^\circ$ is a dense subset of $U$.

By induction for every $n\in\w$ we shall construct a discrete subfamily $\DD_n\subset \K$ and for every $D\in\DD_n$ an open neighborhood $U_n(D)\subset X$ of $D$, and a homeomorphism $h_{n,D}:X\to X$ such that the following conditions are satisfied:
\begin{enumerate}
\item $\DD_n\supset\DD_{n-1}$;
\item $\bigcup\DD_n\subset U$ and for each $u\in U$ there is a point $x\in\bigcup\DD_n$ with $d(x,u)<2^{-n}$;
\item $D\subset U_n(D)\subset U$ for every $D\in\DD_n$;
\item $U_n(D)\subset U_{n-1}(D)\cap O_d(D,2^{-n-1})$ for each $D\in\DD_{n-1}$, and $\diam \big( U_n(D)\big)<2^{-n}$ for every $D\in\DD_n\setminus\DD_{n-1}$;
\item the family $\big(U_n(D)\big)_{D\in\DD_n}$ is discrete in $X$;
\item for each $k<n$, $D\in\DD_k$ and $D'\in \DD_n\setminus \DD_{n-1}$ either  $\bar U_n(D')\cap \bar U_k(D)=\emptyset$ or else $U_n(D')\subset U_k(D)$ and $\diam\big(h_{k,D}(U_n(D'))\big)<2^{-n}$, and
\item $h_{n,D}|X\setminus U_n(D)=\id$ and $\diam(h_{n,D}(D))<2^{-n}$ for each $D\in\DD_n$.
\end{enumerate}
We start the inductive construction by letting $\DD_{-1}=\emptyset$. Assume that for some $n\in\w$ the  families $\DD_k$, neighborhoods $U_k(D)$, $D\in\DD_k$, and homeomorphisms $h_{k,D}$, $D\in\DD_k$, have been constructed for all $k<n$. The inductive assumption (5) implies that the union $B=\bigcup_{k<n}\bigcup_{D\in\DD_k}\partial U_k(D)$ of boundaries of the open sets $U_k(D)$ is a closed nowhere dense subset in $X$.

Consider the subset $V=U\setminus O_d(\bigcup\DD_{n-1},2^{-n})$ and the dense subset $W=V\setminus B$ of $V$. Using Zorn's Lemma, find a maximal subset $S\subset W$, which is $2^{-n-1}$-{\em separated} in the sense that $d(x,y)\ge 2^{-n-1}$ for any distinct points $x,y\in S$.

\begin{claim} For every point $v\in V$ there is a point $s\in S$ such that $d(s,v)<\frac34\cdot{2^{-n}}$.
\end{claim}

\begin{proof} Assume that $d(v,s)\ge \frac34\cdot {2^{-n}}$ for all $s\in S$. Then for any point $w\in O_d(v,2^{-n-2})\setminus B$ and each $s\in S$ we get $d(w,s)\ge d(v,s)-d(v,w)> \frac34{2^{-n}}-\frac14{2^{-n}}=2^{-n-1}$. Consequently, the set $S\cup\{w\}\subset W$ is $2^{-n-1}$-separated, which contradicts the maximality of $S$.
\end{proof}

For each point $s\in S$ chose a positive number $\e_s<2^{-n-3}$ such that for the open $\e_s$-ball $U_s=O_d(s,\e_s)$ and any $k<n$ and $D\in\DD_k$ the following conditions hold:
\begin{itemize}
\item $U_s\subset W$;
\item if $s\in U_k(D)$, then $\overline{U}_s\subset U_k(D)$ and $\diam(h_{k,D}(U_s))<2^{-n}$, and
\item if $s\notin U_k(D)$, then $\bar U_n(D')\cap \bar U_k(D)=\emptyset$.
\end{itemize}
By Definition~\ref{d:K-tame}, we can find in each ball $U_s$ a set $K_s\in\K$. Put $\DD_n=\DD_{n-1}\cup\{K_s,s\in S\}$. The choice of the set $S$ and the numbers $\e_s$, $s\in S$, guarantees that the family $\DD_n$ is discrete in $X$ and satisfies the conditions (1) and (2) of the inductive construction.
For each $D\in\DD_n$ put $U_n(D)=U_s$ if $D=K_s$ for some $s\in S$ and $U_n(D)=O_d(D,2^{-n-1})\cap U_{n-1}(D)$ if $D\in\DD_{n-1}$. It is easy to see that the family $\big(U_n(D)\big)_{D\in\DD_n}$ satisfies the conditions (3)--(6) of the inductive construction.
Since each set $D\in\DD_n\subset\K$ is locally shrinkable, there is a homeomorphism $h_{n,D}:X\to X$ satisfying the condition (7) of the inductive construction.
This completes the inductive step.

After completing the inductive construction, we obtain a disjoint subfamily $\DD_\w=\bigcup_{n\in\w}\DD_n\subset\K$ inducing the decomposition
$$\DD=\DD_\w\cup\big\{\{x\}:x\in X\setminus\textstyle{\bigcup}\DD_\w\big\}$$of $X$. Taking into account that the family $\K\supset\DD_\w$ does not contains singletons, we conclude that $\DD^\circ=\DD_\w\subset\K$.  The condition (2) of the inductive construction guarantees that the union $\bigcup\DD^\circ=\bigcup_{n\in\w}\DD_n$ is dense in $U$.

\begin{claim} The decomposition $\DD$ is vanishing.
\end{claim}

\begin{proof} Given an open cover $\U$ of $X$ we need to check that the subfamily $$\DD'=\{D\in\DD:\forall U\in\U\;\;D\not\subset U\}$$ is discrete in $X$. This will follow as soon as for each point $x\in X$ we find a neighborhood $O_x\subset X$ of $x$ that meets at most one set $D\in\DD'$.
Find $n\in\w$ such that the ball $O_d(x,2^{-n})$ is contained in some set $U\in\U$. We claim that the family $\DD_x=\{D\in\DD':D\cap O_d(x,2^{-n-1})\ne\emptyset\}$ lies in $\DD_n$. Assume for a contradiction that the family $\DD_x$ contains some set $D\in\DD'\setminus\DD_n$. Then  $\diam(D)<2^{-n-1}$ by the condition (4) of the inductive construction. Taking into account that $D\cap O_d(x,2^{-n-1})\ne\emptyset$, we conclude that $D\subset O_d(x,2^{-n})\subset U$, which contradicts $D\in\DD'$. So, $\DD_x\subset\DD_n$. Since the family $\DD_n$ is discrete in $X$, the point $x$ has a neighborhood $O_x\subset O_d(x,2^{-n-1})$ that meets at most one set of the family $\DD_n$. Then the neighborhood $O_x$ meets at most one set of the families $\DD_x$ and $\DD'$, thus showing that the family $\DD'$ is discrete in $X$ and $\DD$ is vanishing.
\end{proof}

To complete the proof of Theorem~\ref{t:exist}, it remains to check that the decomposition $\DD$ is strongly shrinkable. Given a $\DD$-saturated open subset $W\subset X$, a $\DD$-saturated open cover $\U$ of $W$, and an open cover $\V$ of $W$, we need to construct a homeomorphism $h:W\to W$ such that $(h,\id)\prec\U$ and $\{h(D):D\in\W,\;D\subset W\}\prec\V$.

\begin{claim}\label{cl:sv} The family $\DD'=\{D\in\DD:D\subset W,\;\;\forall V\in\V\;\;D\not\subset V\}$ is discrete in $W$.
\end{claim}

\begin{proof} Assuming that the disjoint family $\DD'$ is not discrete in $W$, find a point $x\in W$ such that each neighborhood $O_x\subset W$ meets infinitely many sets of the family $\DD'$. By the regularily of the metrizable space $X$, the point $x$ has a closed  neighborhood $N_x\subset X$ such that $N_x\subset W$. Then the open cover $\V_X=\V\cup\{X\setminus N_x\}$ witnesses that the decomposition $\DD$ is not vanishing in $X$, which is a desired contradiction.
\end{proof}

By Claim~\ref{cl:sv}, the family $\DD'$ is discrete in $W$. Consequently, for each set $D\in\DD'$ we can find an open neighborhood $O(D)\subset W$ such that the family $\{O(D)\}_{D\in\DD'}$ is discrete in $W$.
Since each set $D\in\DD'$ is compact, we can find a number
$n_D\in\IN$ so large that
\begin{itemize}
\item $D\in \DD_{n_D}$;
\item $O_d(D,2^{-n_D})\subset O(D)\cap U$ for some $\DD$-saturated open set $U\in\U$, and
\item each subset $B\subset O_d(D,2^{-n_D})$ of diameter $\diam(B)<2^{-n_D}$ lies in some set $V\in\V$.
\end{itemize}
Now consider the homeomorphism $h:W\to W$ defined by
$$h(x)=\begin{cases}h_{n_D,D}(x)&\mbox{if $x\in U_{n_D}(D)$ for some $D\in\DD'$}\\
x&\mbox{otherwise}
\end{cases}
$$The conditions (4) and (7) of the inductive construction and the choice of the numbers $n_D$, $D\in\DD'$, guarantee that $h$ is a well-defined homeomorphism of $W$ with $(h,\id_W)\prec\U$.

Next we show that for each set $K\in\DD$ the image $h(K)$ lies in some set $V\in\V$. This is clear if $K$ is a singleton. So, assume that the set $K\in\DD$ is not a singleton.
If $K=D$ for some $D\in\DD'$, then $\diam(h(K))=\diam(h(D))=\diam(h_{n_D}(D))<2^{-n_D}$ by the condition $(7)$ of the inductive assumption and hence $h(D)\subset V$ for some set $V\in\V$ by the choice of the number $n_D$.
Next, assume that $K\notin\DD'$. Find a unique number $k\in\w$ such that $K\in\DD_k\setminus\DD_{k-1}$. If $K\subset U_{n_D}(D)$ for some $D\in\DD'$, then $k>n_{D}$ by the condition (5) of the inductive construction, and the set $h(K)=h_{n_D,D}(K)$ has diameter $\diam(h(K))<2^{-n_{D}}$ by the condition (6) of the inductive construction.

If $K\not\subset U_{n_D}(D)$ for all $D\in\DD'$, then $K$ is disjoint with the union $\bigcup_{D\in\DD'}U_{n_{D}(D)}$ by the condition (6) of the inductive construction and then $h(K)=K\subset V$ for some $V\in\V$ by the definition of the family $\DD'\ni K$.

\section{Topological equivalence and universality of $\K$-spongy sets}

In this section we shall derive from Corollary~\ref{c2.6} a general version of Theorem~\ref{main1} treating so-called $\K$-spongy sets.

\begin{definition} Let $\K$ be a tame family of compact subsets of a topological space $X$ such that each set $K\in\K$ has non-empty interior $\Int(K)$ in $X$. A subset $S\subset X$ is called {\em $\K$-spongy} if there is a dense $\K$-tame decomposition $\DD$ of $X$ such that $X\setminus S=\bigcup\{\Int(D):D\in\DD\}$.
\end{definition}

Theorem~\ref{main1} will be derived from the following more general theorem.

\begin{theorem}\label{main1K} Let $X$ be a strongly locally homogeneous completely metrizable space, and $\K$ be a tame family of compact subsets  $X$ such that each set $K\in\K$ contains more than one point and has a non-empty interior in $X$. Then:
\begin{enumerate}
\item Each nowhere dense subset of $X$ lies in a $\K$-spongy subset of $X$;
\item Any two $\K$-spongy subsets of $X$ are ambiently homeomorphic, and
\item Any $\K$-spongy subset of $X$ is a universal nowhere dense subset in $X$.
\end{enumerate}
\end{theorem}

\begin{proof} 1. Given a nowhere dense subset $A\subset X$, consider the open dense subset $W=X\setminus\bar A$, and using Theorem~\ref{t:exist}, find a $\K$-tame decomposition $\DD$ of $X$ such that $\bigcup\DD^\circ$ is a dense subset of $W$. Then $\DD$ is a dense $\K$-tame decomposition and $S=X\setminus\bigcup_{D\in\DD}\Int(D)$ is a $\K$-spongy set containing the nowhere dense set $A$.
\smallskip

2. Given two $\K$-spongy sets $S$ and $S'$ in $X$, find dense $\K$-tame decompositions $\DD$ and $\DD'$ of $X$ such that $X\setminus S=\bigcup_{D\in\DD}\Int(D)$ and  $X\setminus S'=\bigcup_{D\in\DD'}\Int(D)$. By Corollary~\ref{c2.6}, the decompositions $\DD$ and $\DD'$ are topologically equivalent. Consequently, there is a $(\DD,\DD')$-factorizable homeomorphism $\Phi:X\to X$, which maps $X\setminus S$ onto $X\setminus S'$ and witnesses that the $\K$-spongy sets $S$ and $S'$ are ambiently homeomorphic.
\smallskip

3. The third statement of Theorem~\ref{main1K} follows immediately from the first two statements of this theorem.
\end{proof}

\section{Spongy sets in Hilbert cube manifolds}\label{s:spongeQ}

In this section we shall prove Theorem~\ref{t:spongeQ}. Given a spongy subset $S$ in a Hilbert cube manifold $M$, we need to prove that $S$ is a retract of $M$, homeomorphic to $M$. Let $d$ be any metric generating the topology of the space $M$.

 Let $\C$ be the family of connected components of the complement  $M\setminus S$.
Since $M$ is a spongy set, the closure $\bar C$ of each set $C\in\C$ is a tame ball in the Hilbert cube manifold $M$.
This implies that the pair $(\bar C,\partial C)$ is homeomorphic to $(\II^\w\times[0,1],\II^\w\times\{1\})$. Here by $\partial C$ we denote the boundary of $C$ in $M$.  So, we can choose a retraction $r_C:\bar C\to\partial C$ such that the preimage $r_C^{-1}(y)$ of each point $y\in \partial C$ is homeomorphic to the closed interval $\II=[0,1]$. Extend the retraction $r$ to a retraction $\bar r_C:M\to M\setminus C$ defined by $\bar r|\bar C=r_C$ and $\bar r|M\setminus C=\id$. The vanishing property of the family $\C$ guarantees that the map $r:M\to M\setminus\bigcup \C$ defined by
$$r(x)=\begin{cases}
r_C(x)&\mbox{if $x\in \bar C$ for some $C\in\C$,}\\
x&\mbox{otherwise}
\end{cases}
$$is a continuous retraction of $M$ onto the spongy set $S=M\setminus\bigcup\C$ such that the preimage of each point $y\in S$ is either a singleton or an arc. Being a retract of the Hilbert cube manifold $M$, the spongy set $S$ is a locally compact ANR.

\begin{claim} The spongy set $S$ is a Hilbert cube manifold.
\end{claim}

\begin{proof} According to the characterization theorem of Toru\'nczyk \cite{Tor80}, it suffices to show that for each $\epsilon>0$ and a continuous map $f:\II^\w\times\{0,1\}\to S$ there is a continuous map $\tilde f:\II^\w\times \{0,1\}\to X$ such that $d(\tilde f,f)<\e$ and  $\tilde f(\II^\w\times\{0\})\cap \tilde f(\II^\w\times\{1\})=\emptyset$.

Since $M$ is an $\II^\w$-manifold, by Theorem 18.2 of \cite{Chap}, the map $f:\II^\w\times \{0,1\}\to S\subset M$ can be approximated by a map $g:\II^\w\times \{0,1\}\to M$ such that
$d(g,f)<\frac12\epsilon$ and $g(\II^\w\times\{0\})\cap g(\II^\w\times\{1\})=\emptyset$. Fix a positive real number $\delta<\epsilon$ such that $$\delta\le \dist\big(g(\II^\w\times\{0\}),g(\II^\w\times\{1\})\big)=\inf\big\{d(x,y):x\in g(\II^\w\times\{0\}),\;\;y\in g(\II^\w\times\{1\})\big\}.$$

The vanishing property of the family $\C$ guarantees that the subfamily $\C'=\{C\in\C:\diam(C)\ge\delta/5\}$ is discrete in $M$. By the collectivewise normality of $M$, for each set $C\in\C'$ its closure $\bar C$ has an open neighborhood $O(\bar C)\subset M$ such that the indexed family $\big(O(\bar C)\big)_{C\in\C'}$ is discrete in $X$. Since for each set $C\in\C'$ the closure $\bar C$ is a tame ball in $M$, we can additionally assume that the pair $(O(\bar C),\bar C)$ is homeomorphic to the pair $\big(\II^\w\times[0,2),\II^\w\times[0,1]\big)$.

\begin{claim}\label{cl10.2} For every $C\in\C'$ there is a map $g_C:\II^\w\times\{0,1\}\to M\setminus C$ such that
\begin{enumerate}
\item $d(g_C,\bar r_C\circ g)<\delta/5$;
\item $g_C|g^{-1}(M\setminus O(\bar C))= g|g^{-1}(M\setminus O(\bar C))$;
\item $g_C(g^{-1}(\bar C))\subset\partial C$;
\item $g_C(g^{-1}(O(\bar C)))\subset O(\bar C)$, and
\item $g_C(\II^\w\times\{0\})\cap g_C(\II^\w\times\{1\})=\emptyset$.
\end{enumerate}
\end{claim}

\begin{proof} Choose an open neighborhood $U(\bar C)$ of $\bar C$ in $M$ such that $\bar U(\bar C)\subset O(\bar C)$. Consider the closed subset $F_C=g^{-1}(\bar C)\subset\II^\w\times\{0,1\}$, and its open neighborhoods $O(F_C)=g^{-1}(O(\bar C))$ and $U(F_C)=g^{-1}(U(\bar C))$.
It follows from $\bar U(\bar C)\subset O(\bar C)$ that $\bar U(F_C)\subset O(F_C)$.

Next, consider the map $\bar r_C\circ g|O(F_C):O(F_C)\to O(\bar C)\setminus C$.
Since $O(\bar C)\setminus C$ is an absolute retract (homeomorphic to $\II^\w\times[1,2)$), by Theorems 5.1.1 and 5.1.2 of \cite{Hu}, there is an open cover $\U_C$ of $O(\bar C)\setminus C$ such that any map $g':F_C\to O(\bar C)\setminus C$ with $(g',\bar r_C\circ g|F_C)\prec\U_C$ can be extended to a map $g'_C:O(F_C)\to  O(\bar C)\setminus C$ such that $g'_C|O(F_C)\setminus U(F_C)=g|O(F_C)\setminus (F_C)$ and $d(g'_C,g|O(C))<\epsilon/5$.

Since the boundary $\partial C$ of the tame ball $\bar C$ in $M$ is homeomorphic to the Hilbert cube $\II^\w$, by Theorem~8.1 of \cite{Chap}, the map $\bar r\circ g|F_C\to\partial C$ can be approximated by an injective map $g':F_C\to\partial C$ such that $(g',g|F_C)\prec\U_C$. By the choice of the cover $\U_C$ the map $g'$ can be extended to a continuous map $g'_C:O(F_C)\to O(\bar C)\setminus C$ such that $g'_C|O(F_C)\setminus U(F_C)=g|O(F_C)\setminus U(F_C)$ and $d(g'_C,g|O(F_C))<\delta/5$.

Extend the map $g'_C$ to a continuous map $g_C:\II^\w\times\{0,1\}\to M\setminus C$ such that
$$g_C(x)=\begin{cases}
g'_C(x)&\mbox{if $x\in O(C)$}\\
g(x)&\mbox{otherwise}.
\end{cases}
$$
It is easy to see that the map $g_C$ satisfies the conditions (1)--(5).
\end{proof}

Now define a map $\tilde g:\II^\w\times\{0,1\}\to M'$ by the formula
$$\tilde g(x)=\begin{cases}
g_C(x)&\mbox{if $x\in g^{-1}(O(\bar C))$ for some $C\in\C'$};\\
g(x)&\mbox{otherwise}.
\end{cases}
$$
Claim~\ref{cl10.2} implies that $d(\tilde g,g)<\delta/5$ and $\tilde g(\II^\w\times\{0\})\cap \tilde g(\II^\w\times\{1\})=\emptyset$.
Finally, put $\tilde f=r\circ \tilde g:\II^\w\times\{0,1\}\to S$.

The choice of the family $\C'$ guarantees that $d(\tilde f,\tilde g)<\delta/5$ and hence $d(\tilde f,g)<\frac25\delta$ and $d(\tilde f,f)\le d(\tilde f,g)+d(g,f)<\frac25\delta+\frac12\epsilon<\epsilon$. The choice of $\delta\le\dist\big(g(\II^\w{\times}\{0\}),g(\II^\w{\times}\{0\})\big)$ guarantees that
$$\dist\big(\tilde f(\II^\w{\times}\{0\}),\tilde f(\II^\w{\times}\{0\})\big)\ge\delta-2d(\tilde f,g)\ge\frac15\delta>0$$and thus $\tilde f(\II^\w{\times}\{0\})\cap\tilde f(\II^\w{\times}\{1\})=\emptyset$. By the characterization theorem of Toru\'nczyk \cite{Tor80}, the space $S$ is an $\II^\w$-manifold.
\end{proof}

Since for each point $y\in S$ the preimage $r^{-1}(y)$ is either a singleton or an arc, the retraction $r:M\to S$ is a cell-like surjective map between Hilbert cube manifolds $M$ and $S$. By Corollary 43.2 of \cite{Chap}, the map $r$ is a near homeomorphism. So, the Hilbert cube manifolds $M$ and $S$ are homeomorphic.

\section{The family of tame balls in a manifold is tame}

In this section we shall show that the family $\K$ of tame balls in an $\II^n$-manifold $X$ is tame and each vanishing decomposition $\DD\subset\K\cup\big\{\{x\}:x\in X\big\}$ of $X$ is $\K$-tame.

\begin{theorem}\label{t:tame} Let $n\in\IN\cup\{\w\}$ and $X$ be an $\II^n$-manifold. Then:
\begin{enumerate}
\item The family $\K$ of tame balls in $X$ is tame.
\item Each vanishing decomposition $\DD\subset\K\cup\big\{\{x\}:x\in X\big\}$ of $X$ is strongly shrinkable and hence is $\K$-tame.
\end{enumerate}
\end{theorem}

\begin{proof} (1) The definition of a tame ball implies that the family $\K$ is ambiently invariant.
If $X$ is a finite dimensional manifold, then the local shift property of $\K$ follows from the Annulus Conjecture proved in \cite{Rado}, \cite{Moise}, \cite{Quinn,Edwards}, \cite{Kirby} for dimensions 2, 3, 4, and $\ge 5$, respectively. If $X$ is a Hilbert cube manifold, then the local shift property follows from Theorem 11.1 of \cite{Chap} on extensions of homeomorphisms between $Z$-sets of the Hilbert cube.

The strong shrinkability of tame balls in finite dimensional manifolds was proved in Proposition 6.2 \cite{Dav}. The strong shrinkability of tame balls in Hilbert cube manifolds follows from Theorem 2.4 of \cite{Cerin} and Corollary 43.2 of \cite{Chap}. The fact that each non-empty open subset of the manifold $X$ contains a tame ball is trivial if $X$ is finite dimensional and follows from 12.1 \cite{Chap} if $X$ is a Hilbert cube manifold.
\smallskip

(2) Let $\DD\subset\K\cup\big\{\{x\}:x\in X\big\}$ be a vanishing decomposition of the manifold $X$ into singletons and tame balls. If $X$ is finite dimensional, then each tame ball $D\in\DD^\circ$ has a neighborhood homeomorphic to $\IR^n$ and hence $D$ does not intersect the boundary $\partial X$ of the manifold $X$. Then $\DD$ is a vanishing decomposition of the $\IR^n$-manifold $M=X\setminus \partial X$. By Theorem~8.7 of \cite{Dav}, it is strongly shrinkable. If $X$ is a Hilbert cube manifold, then the strong shrinkability of the decomposition $\DD$ follows from Corollary 43.2 \cite{Chap} (saying that each cell-like map between Hilbert cube manifolds is a near homeomorphism), and Theorem 5.3 of \cite{Cerin} implying the decomposition space $X/\DD$ is a Hilbert cube manifold. The latter fact can be alternatively deduced from Theorem~\ref{t:spongeQ} and Toru\'nczyk's Theorem 3' (saying that for a decomposition $\DD$ of an $\II^\w$-manifold $M$ the decomposition space $M/\DD$ is an $\II^\w$-manifold provided the union $\bigcup\DD^\circ$ is contained in a countable union of $Z$-sets in $M$).
\end{proof}

\section{Proof of Theorem~\ref{main1}}\label{s:main1}

Given an $\II^n$-manifold $X$ we need to prove the following statements:
\begin{enumerate}
\item Each nowhere dense subset of $X$ lies in a spongy subset of $X$;
\item Any two spongy subsets of $X$ are ambiently homeomorphic, and
\item Any spongy subset of $X$ is a universal nowhere dense subset in $X$;
\end{enumerate}
By Theorem~\ref{t:tame}, a subset $S\subset X$ is spongy if and only if $S$ is $\K$-spongy for the family $\K$ of tame balls in $X$. If $X$ is a Hilbert cube manifold, then $X$ is a strongly locally homogeneous completely metrizable space and the statements (1)--(3) follow immediately from Theorem~\ref{main1K}.

The same argument works if $X$ is an $\IR^n$-manifold for a finite $n$. It remains to consider the case of an $\II^n$-manifold $X$ that has non-empty boundary $\partial X$ (which consists of points $x\in X$ that do not have open neighborhoods homeomorphic to $\IR^n$).
It follows that $M=X\setminus\partial X$ is an $\IR^n$-manifold. Theorem~\ref{t:tame} guarantees that the family $\K(M)$ of tame balls in $M$ is tame.

By Theorem~\ref{t:exist}, each nowhere dense subset of $X$ is contained in a $\K$-spongy subset of $X$, so the statement (1) holds for the $\II^n$-manifold $X$.

To prove the statement (2), fix any two spongy subsets $S$ and $S'$ in $X$. Denote by $\C$ and $\C'$ the families of connected components of the complements $X\setminus S$ and $X\setminus S'$.
By the definition of a spongy set, for each component $C\in\C$ its closure $\bar C$ is a tame ball in $X$ and hence $\bar C$ has an open neighborhood homeomorphic to $\IR^n$. Then $\bar C\cap\partial X=\emptyset$ and hence $\bar C\subset M$. Now consider the dense decompositions
$$
\begin{aligned}
\A&=\{\bar C:C\in\C\}\cup\big\{\{x\}:x\in M\setminus\bigcup_{C\in\C}\bar C\big\} \mbox{ and }\\
\BB&=\{\bar C:C\in\C'\}\cup\big\{\{x\}:x\in M\setminus\bigcup_{C\in\C'}\bar C\big\}
\end{aligned}
$$of the $\IR^n$-manifold $M$. The vanishing property of the families $\C$ and $\C'$ implies that the decompositions $\A$ and $\BB$ of the $\IR^n$-manifold $M$ are vanishing and hence $\K(M)$-tame according to Theorem~\ref{t:tame}.

Fix any metric $d$ generating the topology of the manifold $X$ and by the paracompactness of $X$, find an open cover $\U$ of $X$ such that $\St(x,\U)\subset O_d(x,d(x,\partial X)/2)$ for each point $x\in M$.
By (the proof of ) Corollary~\ref{c2.6}, there is a homeomorphism $\Phi:M\to M$ such that $\{\Phi(A):A\in\A\}=\BB$ and for each point $x\in M$ there are sets $A\in\A$ and $B\in\BB$ such that $x\in\St(A,\U)$, $\Phi(x)\in\St(B,\U)$ and $\St(A,\U)\cap\St(B,\U)\ne\emptyset$.

Extend the homeomorphism $\Phi:M\to M$ to a bijective map $\bar \Phi:X\to X$ such that $\bar\Phi|M=\Phi$ and $\bar\Phi|\partial X=\id$. We claim that the functions $\bar\Phi$ and $\bar\Phi^{-1}$ are continuous. It is necessary to check the continuity of these functions at each point $x_0\in\partial X$. First we verify the continuity of the function $\bar\Phi$ at $x_0$.
Given any $\epsilon>0$ we need to find $\delta>0$ such that $\bar\Phi(O_d(x_0,\delta))\subset O_d(x_0,\e)$.

Repeating the proof of Claim~\ref{cl:ed}, for the number $\epsilon$, we can find a positive real number $\eta\le\epsilon$ such that for each set $B\in \BB$ with $\St(B,\U)\cap O_d(x_0,\eta)\ne\emptyset$, we get $\St(B,\U)\subset O_d(x_0,\epsilon)$. Next, by the same argument, for the number $\eta$ choose a positive real number $\delta\le\eta$ such that for each set $A\in \A$ with $\St(A,\U)\cap O_d(x_0,\delta)\ne\emptyset$, we get $\St(A,\U)\subset O_d(x_0,\eta)$.

We claim for each point $x\in X$ with $d(x,x_0)<\delta$, we get $d(\bar\Phi(x),x_0)<\e$. This inequality trivially holds if $x\in\partial X$. So, we assume that $x\in M$.
By the choice of the homeomorphism $\Phi$, there are sets $A\in\A$ and $B\in\BB$ such that $x\in\St(A,\U)$, $\Phi(x)\in\St(B,\U)$, and the intersection $\St(A,\U)\cap\St(B,\U)$ contains some point $y\in X$. Taking into account that the set $\St(A,\U)$ meets the ball $O_d(x_0,\delta)\ni x$, we conclude that $y\in\St(A,\U)\subset O_d(x_0,\eta)$. Since the set $\St(B,\U)\ni y$ meets the ball $O_d(\eta)$, the choice of the number  $\eta$ guarantees that $\bar\Phi(x)=\Phi(x)\in\St(B,\U)\subset O_d(x_0,\epsilon)$.
This means that the map $\bar\Phi$ is continuous.

By analogy we can show that the inverse map $\bar\Phi^{-1}:X\to X$ is continuous too.
So, $\bar\Phi:X\to X$ a homeomorphism of $X$ such that $\Phi(\bigcup_{C\in\C'}\bar C)=\Phi(\bigcup\A^\circ)=\bigcup\BB^\circ=\bigcup_{C\in\C'}\bar C$.
This implies that $\Phi(\bigcup\C)=\bigcup\C'$ and $\Phi(S)=\Phi(X\setminus \bigcup\C)=X\setminus\bigcup\C'=S'$, witnessing that the spongy sets $S$ and $S'$ are ambiently homeomorphic in $X$. This completes the proof of the statement (2) of Theorem~\ref{main1}.

The statement (3) follows immediately from the statements (1) and (2).

\section{Topological equivalence of cellular decompositions of Hilbert cube manifolds}\label{s:cellular}

In this section we shall apply Theorem~\ref{main2} to prove topological equivalence of certain cellular decompositions of Hilbert cube manifolds. But first we shall study the structure of tame families of compact subsets in more general topological spaces.

The following proposition shows that for strongly locally homogeneous completely metrizable spaces the definition~\ref{d:K-tame} of a tame family can be  a bit simplified.

\begin{proposition}\label{p4.1} Let $X$ be a strongly locally homogeneous completely metrizable space and $\K$ be an ambiently invariant family of locally shrinkable compact subsets of $X$, possessing the local shift property. Then the following conditions are equivalent:
\begin{enumerate}
\item $\bigcup\K=X$;
\item $\bigcup\K$ is dense in $X$;
\item each non-empty open set $U\subset X$ contains a set $K\in\K$, and
\item for each point $x\in X$ and each open neighborhood $U\subset X$ of $x$ there is a set $K\in\K$ such that $x\in K\subset U$.
\end{enumerate}
\end{proposition}

\begin{proof} It is clear that $(4)\Ra(3)\Ra(2)\Leftarrow (1)\Leftarrow(4)$. So, it remains to prove the implication $(2)\Ra(4)$. Given a point $x\in X$ and an open neighborhood $U_x\subset X$ of $x$, consider the orbit $O_x=\{h(x):h\in\HH(X)\}$ of $x$ under the action of the homeomorphism group $\HH(X)$ of $X$. The strong local homogeneity of $X$ implies that this orbit is open-and-closed in $X$. Since the union $\bigcup\K$ of the family $\K$ is dense in $X$, there exists a set $K'\in\K$ that intersects the orbit $O_x$. So, there exists a homeomorphism $f:X\to X$ such that $f(x)\in K'$. Then the compact set $K=f^{-1}(K')$ contains the point $x$ and belongs to the family $\K$ (by the ambient invariance of $\K$).

Since the set $K\in\K$ is locally shrinkable, the quotient map $q_K:X\to X/K$ is a strong near homeomorphism by Theorem~\ref{t:shrink}, which implies that the space $X/K$ is homeomorphic to $X$ and hence is strongly locally homogeneous. Then for the point $y=q_K(K)\in X/K$ its orbit $O_y$ under the action of the homeomorphism group $\HH(X/K)$ is closed-and-open in the quotient space $X/K$.
Then $W=q_\DD^{-1}(O_y)$ is a closed-and-open neighborhood of $K$ in $X$.
Since the quotient map $q_K:X\to X/K$ is a strong near homeomorphism, there is a homeomorphism $h_1:X\to X/K$ such that $h_1|X\setminus W=q_K|X\setminus W$ and hence $h_1(W)=O_y$. Since $h_1(x)\in O_y$, there is a homeomorphism $h_2:X/K\to X/K$ such that $h_2(h_1(x))=y$.

Since the space $X/K$ is strongly locally homogeneous, for the neighborhood $U_y=h_2\circ h_1(U_x)\cap O_y$ of the point $y=q_K(K)$ there is a neighborhood $V_y\subset U_y$ such that for any point $z\in V_y$ there is a homeomorphism $h:X/K\to X/K$ such that $h(z)=y$ and $h(U_y)=U_y$.

Since $q_K$ is a strong near homeomorphism, for the neighborhood $V_y$ of the point $y=q_K(K)$ there is a homeomorphism $h_3:X\to X/K$ such that $h_3(K)\subset V_y$. By the choice of $V_y$ for the point $z=h_3(x)\in h_3(K)\subset V_y$ there is a homeomorphism $h_4:X/K\to X/K$ such that $h_4(z)=y$ and $h_4(U_y)=U_y$. Then the homeomorphism $h=h_1^{-1}\circ  h_2^{-1}\circ h_4\circ h_3:X\to X$ has the properties:
$$h(x)=h_1^{-1}\circ  h_2^{-1}\circ h_4\circ h_3(x)=h_1^{-1}\circ  h_2^{-1}\circ h_4(z)=h_1^{-1}\circ  h_2^{-1}(y)=h_1^{-1}(h_1(x))=x$$and
$$h(K)=h_1^{-1}\circ  h_2^{-1}\circ h_4\circ h_3(K)\subset h_1^{-1}\circ  h_2^{-1}\circ h_4(U_y)=
 h_1^{-1}\circ  h_2^{-1}(U_y)=U_x.$$
Since the family $\K$ is ambiently invariant, the compact set $h(K)$ belongs to the tame family $\K$ and has the required properties: $x=h(x)\in h(K)\subset U_x$.
\end{proof}

\begin{proposition}\label{p4.2} If $\K$ is an ambiently invariant  family of locally shrinkable compact subsets of a topologically homogeneous completely metrizable space $X$ and $\K$ has the local shift property, then any two sets $A,B\in\K$ are ambiently homeomorphic.
\end{proposition}

\begin{proof} By Theorem~\ref{t:shrink}, the quotient maps $q_A:X\to X/A$ and $q_B:X\to X/B$ are strong near homeomorphisms. This implies that the decomposition spaces $X/A$ and $X/B$ are homeomorphic to $X$ and hence are topologically homogeneous. So, we can choose a homeomorphism $f:X/A\to X/B$ that maps the singleton $\{A\}=q_A(A)\in X/A$ onto the singleton  $\{B\}=q_B(B)\in X/B$.

Since the quotient space $X/B$ is homeomorphic to $X$, we can consider the ambiently invariant family $\K(X/B)=\{h(K):K\in\K,\;h\in\HH(X,X/B)\}$ of compact subsets of $X/B$ induced by the tame family $\K$. Since this family has the local shift property, the point $B\in X/B$ has a neighborhood $U\subset X/B$ such that for any two compact sets $K,K'\in\K(X/B)$ in $U$ there is a homeomorphism $h:X/B\to X/B$ such that $h(K)=K'$. Since the quotient maps $q_B:X\to X/B$ and $q_A:X\to X/A$ are strong near homeomorphisms, there are homeomorphisms $h_B:X\to X/B$  and $h_A:X\to X/A$ such that $h_B(B)\subset U$ and $h_A(A)\subset f^{-1}(U)$. Then the compact sets $K=f\circ h_A(A)$ and $K'=h_B(B)$ belong to the family $\K(X/B)$ and lie in the open set $U\subset X/B$. By the choice of $U$, there is a homeomorphism $h:X/B\to X/B$ such that $h(K)=K'$. Now we see that the homeomorphism $h^{-1}_B\circ h\circ f\circ h_A:X\to X$ maps $A$ onto $B$, and hence the sets $A$ and $B$ are ambiently homeomorphic.
\end{proof}

Now we consider three shape properties of subsets.
A compact subset $K$ of a topological space $X$ will be called
\begin{itemize}
\item {\em point-like} if for each closed neighborhood $N\subset X$ of $K$ the complement $N\setminus K$ is homeomorphic to the complement $N\setminus\{x\}$ of some interior point $x\in \Int(N)$ of $N$;
\item {\em cell-like} if for each neighborhood $U$ of $K$ in $X$ the set $K$ is contractible in $U$;
\item {\em cellular} if for each neighborhood $U$ of $K$ in $X$ there is a neighborhood $V\subset U$ of $K$ homeomorphic to
    $$\begin{cases}
\IR^n&\mbox{if $n=\dim(X)$ is finite,}\\
\II^\w\times [0,1)&\mbox{if $\dim(X)$ is infinite}.
\end{cases}
$$\end{itemize}
If each singleton $\{x\}\subset X$ of a paracompact topological space is cellular, then $X$ is a manifold modeled on the Hilbert cube $\II^\w$ or an Euclidean space $\IR^n$, where $n=\dim(X)$.

Each cellular subset in an $\II^n$-manifold is cell-like but the converse is not true even for $\IR^n$-manifolds as shown by the Whitehead Example 9.7 \cite{Dav}. On the other hand, cellularity is equivalent to point-likeness, as shown by the following characterization whose finite dimensional case was proved in  \cite[Proposition 2]{Dav} and \cite{ChriOsb}, and infinite dimensional case in \cite{Cerin}.

\begin{proposition}\label{p13.3} Let $X$ be a manifold modeled on a space $E\in\{\II^\w,\IR^n:n\in\IN\}$. For a compact subset $K$ of $X$ the following conditions are equivalent:
\begin{enumerate}
\item $K$ is point-like;
\item $K$ is cellular;
\item for each neighborhood $U\subset X$ of $K$ there is a tame ball $V\subset U$ that contains $K$;
\item $K$ is locally shrinkable, and
\item the quotient map $q_K:X\to X/K$ is a strong near homeomorphism.
\end{enumerate}
\end{proposition}

We recall that a topological space $X$ is called {\em locally contractible} if for each point $x\in X$ and a neighborhood $U\subset X$ of $x$ there is another neighborhood $V\subset U$ of $x$, which is contractible in $U$.

\begin{proposition}\label{p13.4} If $\K$ is a tame family of compact subsets of a metrizable topological space $X$, then each compact set $K\in\K$ is
\begin{enumerate}
\item point-like in $X$ provided $X$ is completely metrizable;
\item cell-like in $X$ provided $X$ is locally contractible, and
\item cellular in $X$ provided $X$ is a manifold modeled on a space $E\in\{\II^\w,\IR^n:n\in\IN\}$.
\end{enumerate}
\end{proposition}

\begin{proof} Fix a compact set $K\in\K$ and a neighborhood $U$ of $K$ in $X$. By Definition~\ref{d:K-tame}, the set $K$ is locally shrinkable.

(1) If $X$ is completely metrizable, then by Theorem~\ref{t:shrink}, the quotient map $q_K$ is a strong near homeomorphism. Consequently, there exists a homeomorphism $f:X\to X/K$ such that $f|X\setminus U=\id$. Consider the point $K\in X/K$ and its image $x=f^{-1}(K)\in U$ under the inverse homeomorphism $f^{-1}:X/K\to X$. It follows that $h=f^{-1}\circ q_K|\bar U\setminus K:\bar U\setminus K\to \bar U\setminus\{x\}$ a homeomorphism, proving that the set $K$ is point-like in $X$.
\smallskip

(2) If the space $X$ is locally contractible, then the locally shrinkable subset $K\subset X$ is cell-like by Theorem~3.5 of \cite{Dav}.
\smallskip

(3) The third statement follows immediately from Proposition~\ref{p13.3}.
\end{proof}

Propositions~\ref{p4.2} and \ref{p13.4} imply that each tame family $\K$ of compact subsets of a topologically homogeneous $\II^n$-manifold consists of pairwise ambiently homeomorphic cellular subsets and hence $\K=\{h(K_0):h\in\HH(X)\}$ for some cellular subset $K_0\subset X$. Now we are going to prove the converse statement: for each cellular subset $K_0$ of a topologically homogeneous Hilbert cube manifold $X$ the family $\K=\{h(K_0):h\in\HH(X)\}$ is tame.

\begin{theorem}\label{t4.7} A family $\K$ of compact subsets of a topologically homogeneous Hilbert cube manifold $X$ is tame if and only if $\K=\{h(K_0):h\in \HH(X)\}$ for some cellular compact subset $K_0\subset X$.
\end{theorem}

\begin{proof} The ``only if'' part follows from Propositions~\ref{p4.2} and \ref{p13.4}.
To prove the ``if'' part, assume that $\K=\{h(K_0):h\in \HH(X)\}$ for some cellular compact subset $K_0\subset X$. It is clear that thus defined family $\K$ is ambiently invariant and $\bigcup\K=X$ is dense in $X$. Since topologically homogeneous manifolds are strongly locally homogeneous, Proposition~\ref{p4.1} implies that each non-empty open subset of $X$ contains a set $K\in\K$. By Proposition~\ref{p13.3}, each cellular subset of $X$ is locally shrinkable. It remains to show that $\K$ has the local shift property. Given a point $x\in X$ and a neighborhood $O_x\subset X$ we need to find a neighborhood $U_x\subset X$ such that for any sets $K,K'\in\K'$ in $U_x$ there is a homeomorphism $h:X\to X$ such that $h|X\setminus O_x=h|X\setminus O_x$. By Theorem~12.1 of \cite{Chap}, the point $x$ of the Hilbert cube manifold $X$ has a neighborhood $U_x\subset O_x$ homeomorphic to $\II^\w\times[0,1)$. We claim that for any two compact subsets $K_1,K_2\in\K$ in $U_x$ there is a homeomorphism $h:X\to X$ such that $h(K_1)=K_2$ and $h|X\setminus O_x=\id$.

For every $i\in\{1,2\}$ fix a homeomorphism $h_i$ of $X$ such that $h_i(K_0)=K_i$. The set $K_0$, being cellular in $X$, lies in the interior of a tame ball $B_0\subset X$ such that $B_0\subset  h_1^{-1}(U_x)\cap h_2^{-1}(U_x)$. Then $B_1=h_1(B_0)$ and $B_2=h_2(B_0)$ are tame balls in $U_x$ and $h_{12}=h_2\circ h_1:X\to X$ is a homeomorphism such that $h_{12}(K_1)=K_2$ and $f(\partial B_1)=\partial B_2$. Since $U_x$ is homeomorphic to $\II^\w\times[0,1)$, the union $B_1\cup B_2$ lies in the interior of some tame ball $B$ in $U_x$. Being tame, the ball $B$ is homeomorphic to the Hilbert cube $\II^\w$ and its boundary $\partial B$ in $X$ is also homeomorphic to the Hilbert cube $\II^\w$. Moreover, $\partial B_0$ is a $Z$-set in $B$ (which means that the identity map $\id:B\to B$ can be uniformly approximated by maps $B\to B\setminus\partial B$).  By the same reason, for every $i\in\{1,2\}$ the boundary $\partial B_i$ of the tame cube $B_i$ is homeomorphic to $\II^\w$ and is a $Z$-set both in $B_i$ and in the complement $B\setminus\Int(B_i)$. Moreover, since the boundary $\partial B_i$ is a retract of the tame ball $B_i$, the complement $B\setminus \Int(B_i)$ is a retract of the tame ball $B$ and hence $B\setminus\Int(B_i)$ is homeomorphic to the Hilbert cube $\II^\w$, being a compact contractible $\II^\w$-manifold; see \cite[22.1]{Chap}. By Theorem 11.1 of \cite{Chap}, the homeomorphism $h_{12}|\partial B_1\cup \id|\partial B:\partial B_1\cup \partial B\to\partial B_2\cup\partial B$ can be extended to a homeomorphism $\bar h_{12}:B\setminus\Int(B_1)\to B\setminus\Int(B_2)$ such that $\bar h_{12}|\partial B_1=h_{12}|B_1$ and $h_{12}|\partial B=\id$. Then the homeomorphism $h:X\to X$ defined by $h|B_1=h_{12}|B_1$, $h|B_1\setminus\Int(B_1)=\bar h_{12}$, and $h|X\setminus\Int(B)=\id$ has the required property: $h(K_1)=K_2$ and $h|X\setminus O_x=\id$.
\end{proof}

A decomposition $\DD$ of an $\II^n$-manifold $X$ will be called {\em cellular} if each set $D\in\DD$ is cellular in $X$. Theorem~\ref{t4.7} and Corollary~\ref{c2.6} imply the following corollaries.

\begin{corollary}\label{c4.8} Two cellular dense vanishing strongly shrinkable decompositions $\A,\BB$ of a Hilbert cube manifold $X$ are topologically equivalent if any two sets $A\in\A^\circ$ and $B\in\BB^\circ$ are ambiently homeomorphic.
\end{corollary}

%\begin{proof} Since any two sets $A\in\A^\circ$ and $B\in\BB^\circ$ are ambiently homeomorphic, $\A\circ\cup\BB^\circ\subset\{h(K_0):h\in\HH(X)\}$ for some cellular set $K_0\subset X$.
%Consequently, the decompositions $\A$ and $\BB$ are $\K$-tame for the family $\K=\{h(K_0):h\in\HH(X)\}$, which is tame by Theorem~\ref{t4.7}. By Corollary~\ref{c2.6}, the decompositions $\A$ and $\BB$ are topologically equivalent.
%\end{proof}

\begin{corollary}\label{c4.9}  Two cellular dense vanishing decompositions $\A,\BB$ of a topologically homogeneous Hilbert cube manifold $X$ are topologically equivalent if any two sets $A\in\A^\circ$ and $B\in\BB^\circ$ are homeomorphic $Z$-sets in $X$.
\end{corollary}

\begin{proof} By Theorem 3' of \cite{Tor80}, the decomposition space $X/\A$ is a Hilbert cube manifold and by Corollary 43.2 \cite{Chap}, the quotient map $q_\A:X\to X/\A$ is a near homeomorphism. By Theorem~\ref{t:shrink}, the decomposition $\A$ is strongly shrinkable.

Next, we show that any two sets $A\in\A$ and $B\in\BB$ are ambiently homeomorphic in $X$. By our assumption, $A$ and $B$ are homeomorphic cellular $Z$-sets in $X$.  Then there is a homeomorphism $h:A\to B$. Being cellular, the compact sets $A,B$ are connected.
Let $X_A$ and $X_B$ be the connected components of $X$ that contain the sets $A,B$, respectively.
Since the space $X$ is topologically homogeneous, there is a homeomorphism $f:X\to X$ such that $f(X_B)=X_A$.  By Theorem 15.3 \cite{Dav}, the maps  $i_A:A\to X_A$ and $f^{-1}\circ h:A\to X_A$ are homotopic (being homotopic to constant maps into the path-connected space $X_A$).
Since $A$ and $f^{-1}\circ h(A)=f^{-1}(B)$ are $Z$-sets in $X_A$, Theorem 19.4 of \cite{Chap}, yields a homeomorphism $\Phi:X\to X$ such that $\Phi|A=f^{-1}\circ h|A$. Then $f\circ\Phi:(X,A)\to (X,B)$ is a homeomorphism of the pairs, witnessing that the sets $A,B$ are ambiently homeomorphic in $X$. By Corollary~\ref{c4.8}, the decompositions $\A$ and $\BB$ are topologically equivalent.
\end{proof}

\begin{remark} Corollary~\ref{c4.8} cannot be generalized to finite dimensional $\IR^n$-manifolds.
Denote by $\HH_+(\IR^2)$ the subgroup of the homeomorphism group $\HH(\IR^2)$, consisting of orientation preserving homeomorphisms of the real plane $\IR^2$. Take any cellular subset $K_0\subset \IR^2$ such that $K_0\ne h(K_0)$ for each orientation reversing homeomorphism $h\in\HH(\IR^2)\setminus\HH_+(\IR^2)$.
Such a set $K_0$ can look as shown on the picture:

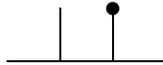
\begin{picture}(100,40)(-100,-10)
\put(20,0){\line(1,0){60}}
\put(40,0){\line(0,1){20}}
\put(60,0){\line(0,1){20}}
\put(60,20){\circle*{5}}
\end{picture}

Consider the families $\K_+=\{h(K_0):h\in\HH_+(\IR^2)\}$ and $\K_-=\{h(K_0):h\in\HH(\IR^2)\setminus\HH_+(\IR^2)\}$. Repeating the proof of Theorem~\ref{t:exist} it is possible to construct dense vanishing strongly shrinkable decompositions $\A$ and $\BB$ of the plane $\IR^2$ such that $$\A^\circ\subset\K_+,\;\;\BB^\circ\subset\K_+\cup\K_-
\mbox{ \ and \ }\BB^\circ\cap\K_+\ne\emptyset\ne\BB^\circ\cap\K_-.$$ It can be shown that the decompositions $\A$ and $\BB$ are not topologically equivalent in spite of the fact that any two sets $A\in\A$ and $B\in\BB$ are ambiently homeomorphic.
\end{remark}

% \section{An Open Problem}

% \begin{problem}
% Construct universal nowhere dense subsets in manifolds modeled on Hilbert spaces, Menger cubes or % N\"obeling spaces.
% \end{problem}

\section*{Acknowledgements}
This research was supported by the Slovenian Research Agency grants P1-0292-0101 and J1-4144-0101.
The first author has been partially financed by NCN means granted by decision DEC-2011/01/B/ST1/01439.

% \newpage

\end{document}